\documentclass[12pt,a4paper]{article}
\usepackage{amsmath,amssymb,amsthm,graphicx,color}
\usepackage[left=1in,right=1in,top=1in,bottom=1in]{geometry}
\usepackage{cite}
\usepackage[british]{babel}
\usepackage{enumerate}
\usepackage{tikz}
\usepackage{comment}
\usepackage{enumitem}
\usepackage{hyperref} 

\hypersetup{
    colorlinks=false,
    pdfborder={0 0 0},
}

\newtheorem{theorem}{Theorem}[section]
\newtheorem{corollary}[theorem]{Corollary}

\newtheorem{lemma}[theorem]{Lemma}

\numberwithin{equation}{section}

\theoremstyle{definition}

\theoremstyle{remark}
\newtheorem{remark}[theorem]{Remark}
\newtheorem{remarks}[theorem]{Remarks}

 \allowdisplaybreaks

\newcommand{\R}{{\mathbb R}}
\newcommand{\Z}{{\mathbb Z}}

\newcommand{\ZP}{{\mathbb Z}_+}
\newcommand{\RP}{{\mathbb R}_+}
\newcommand{\bbX}{{\mathbb X}}

\newcommand{\Exp}{{\mathbb E}}
\renewcommand{\Pr}{{\mathbb P}}
\newcommand{\1}{{\mathbf 1}}

\newcommand{\eps}{\varepsilon}

\newcommand{\re}{{\mathrm{e}}}
\newcommand{\rc}{{\mathrm{c}}}
\newcommand{\ud}{{\mathrm d}}
\newcommand{\F}{{\mathcal F}}

\newcommand{\as}{\ \text{a.s.}}

\newcommand{\io}{\ \text{i.o.}}

\newcommand{\bigmid}{\; \biggl| \;}

\makeatletter
\def\namedlabel#1#2{\begingroup  
    (#2)%
    \def\@currentlabel{#2}%
    \phantomsection\label{#1}\endgroup
}
\makeatother

\title{Non-homogeneous random walks on~a~semi-infinite~strip}

\author{Nicholas Georgiou\footnote{Department of Mathematical Sciences, Durham University, South Road, Durham, DH1 3LE, U.K.}${\ }^{,}$\footnote{{\it Email address:} \url{nicholas.georgiou@durham.ac.uk}} \quad Andrew R. Wade${}^{*,}$\footnote{{\it Email address:} \url{andrew.wade@durham.ac.uk}}}

\begin{document}

\maketitle

\begin{abstract}
We study the asymptotic behaviour of
Markov chains $(X_n,\eta_n)$ on 
$\ZP \times S$, where $\ZP$ is the non-negative integers and $S$ is a finite set.
Neither coordinate is assumed to be Markov. We assume a moments bound on the jumps of $X_n$, and that, roughly speaking,
$\eta_n$ is close to being Markov when $X_n$ is large.
This departure from much of the literature, which assumes that $\eta_n$ is itself a Markov chain,
enables us to probe precisely the recurrence phase transitions by assuming asymptotically zero drift for $X_n$ given $\eta_n$.
We give a recurrence classification in terms of increment moment parameters for $X_n$ and the stationary distribution for the large-$X$
limit of $\eta_n$.  In the null case we also provide a weak convergence result, which demonstrates a form of asymptotic independence between $X_n$ (rescaled) and $\eta_n$. 
Our results can be seen as generalizations of Lamperti's results for non-homogeneous random walks on $\ZP$ (the case where $S$ is a singleton). Motivation arises from modulated queues
or processes with hidden variables where $\eta_n$ tracks an internal state of the system.
\end{abstract}



\section{Introduction}

There are many applications that naturally give rise to Markov processes 
 on a product state-space $\bbX \times S$ where $S$
describes some operating regime or internal state of the system, which 
influences the motion of the process in the primary space $\bbX$. Important classes
of examples include, among others,
\begin{itemize}
\item modulated queues, in which $S$ may contain operating states of the servers or other auxiliary information such as the size of a retrial buffer, as arise in various applications such as those described by Neuts in \cite{neuts};
\item regime-switching processes in mathematical finance or ecology,  where $S$ may contain market or other environmental information;
\item physical processes with internal degress of freedom, where $S$ may describe internal energy or momentum states of a particle,
such as adopted by Sinai as a tool
for studying the Lorentz gas (see e.g.~\cite{ks}), or exemplified by the so-called correlated or persistent random walk.
\end{itemize}
In several of the key examples, the $S$-component of the process is `hidden', and the main interest is in the asymptotic behaviour
of the $\bbX$-component of the process.

In the most classical setting, the projection of the process onto $S$ is itself Markovian. 
In this case, the queueing models become \emph{Markov-modulated} \cite{neuts}, while other examples
fit into the class of \emph{Markov random walks} \cite{hp}.
This case also includes
processes that can be represented as \emph{additive functionals of Markov chains} \cite{rogers}. Such models pose a variety
of mathematical questions, which have been studied rather deeply over several decades using various techniques that take advantage
of the additional Markov structure, and much is now known. 
 
Much less is known when the process projected onto $S$ is \emph{not} Markovian: the main focus
of the present work is to replace the Markovian assumption by a weaker (asymptotic)
condition that provides sufficient structure. This 
 relaxation is necessary to probe
more intimately the recurrence-transience phase transition for these models, since the natural
setting (paralleling the classical work of Lamperti) is to suppose that the law of the process
is \emph{non-homogeneous} in $\bbX$, in particular, the mean drift of the $\bbX$-component
of the process will be asymptotically zero. This non-homogeneity precludes, in general, the $S$-component
of the process from being Markovian, but admits our weaker conditions. 

To avoid technicalities, yet provide a setting rich enough to explore
many interesting phenomena, we take $\bbX$ to be the countable set $\ZP := \{0,1,2,\ldots\}$
and take $S$ to be finite. These models are already of interest for numerous applications,
and there is an existing literature devoted to random walks on half strips ($\ZP \times \{0,1,\ldots, m\}$) or strips ($\Z \times \{ 0,1,\ldots,
m\}$): see \cite{fmm,malyshev,falin1,falin2} and references
therein.  

As an example consider the following queueing model.  
A queue is served by a single
server and experiences arrivals at rate $\lambda$; the service rate is modulated via an internal state of the server $\eta_n$, as well as the length of the queue $X_n$
(in discrete time, i.e., in terms of the jump process).  Allowing the service rate to depend on the queue length distinguishes this model from the class of \emph{semi-Markov} queues~\cite{neuts}. 
 When $(X_n,\eta_n) = (x,i)$, $x \geq 1$, the service rate is $\rho_i (x) = \rho \left( 1 - \frac{2c_i}{x} \right)$, where $c_i$, $i\in S$ 
are parameters of model with $| c_i | < 1/2$. In the case where $c_i \equiv 0$ for all $i$, the internal states of the server are indistinguishable and the model is simply (the jump process of)
 an $M/M/1$
queue with arrival rate $\lambda$ and service rate $\rho$; the critical case from the point of view of recurrence and transience is $\rho =\lambda$,
and so that is the most interesting setting to perturb with non-zero $c_i$. So we take $\rho =\lambda$ from now on. The specification of the model
is completed by stipulating that 
whenever an arrival (departure) occurs the internal state of the server transitions
according to the stochastic matrix $(a_{ij})$ ($(b_{ij})$). In other words, given $(X_n,\eta_n) = (x,i)$, $x \geq 1$,  
\[
(X_{n+1},\eta_{n+1}) = \begin{cases}
(x+1,j) &\text{with probability } \frac{1}{2 (1- \frac{c_i}{x}) } a_{ij} ; \\[1em]
(x-1,j) &\text{with probability } \frac{1- \frac{2c_i}{x} }{2 (1- \frac{c_i}{x}) } b_{ij} .
\end{cases}
\]
Given $(X_n, \eta_n ) = (0,i)$, $(X_{n+1} , \eta_{n+1} ) = (1, j)$ with probability $a_{ij}$.

In general, $(\eta_n)$ is not itself a Markov chain, so this model falls outside the
usual Markov-modulated queue framework. However, for large queue lengths the probabilities of 
arrival and 
 departure are approximately equal, and so
the $\eta_n$ process should be well approximated by the Markov chain on $S$ with transition matrix $M_{ij} =
\frac{1}{2}(a_{ij} +b_{ij})$.  Under the condition that the matrix $M$ be irreducible, our results
determine conditions for transience and recurrence in terms of the stationary distribution of the
chain with transition matrix $M$ and the constants $c_i$.

\section{Model and main results}
\label{sec:model}

We now describe precisely our model. Our state-space is the \emph{half-strip} $\ZP \times S$, where $S$ is finite and nonempty;
for $k \in S$, we call the subset $\ZP \times \{ k \}$ a \emph{line}.
We consider an irreducible Markov chain $(X_n,\eta_n) \in \ZP \times S$, with transition probabilities
\begin{equation}
\label{eqn:transition_probs}
\Pr[(X_{n+1},\eta_{n+1})=(y,j) \mid (X_n,\eta_n) = (x,i)] = p(x,i,y,j),
\end{equation}
and provide conditions for recurrence/transience of $(X_n)$,
in a sense that we explain below. Throughout we use the notation $\F_n := \sigma(X_0,\eta_0,\ldots, X_n,\eta_n)$
and $\RP := [0,\infty)$.

 The process $(X_n)$ is typically not itself a Markov chain; under our standing assumptions,  however, it does
inherit the recurrence/transience dichotomy from $(X_n,\eta_n)$, as the following result shows.
\begin{lemma}\label{Xn-recur-trans}
Exactly one of the following holds:
\begin{enumerate}
[label=(\roman{*})]
\item If $(X_n,\eta_n)$ is recurrent, then $\Pr{[X_n = 0 \io ]} = 1$. 
\item If $(X_n,\eta_n)$ is transient, then $\Pr{[X_n = 0 \io ]} = 0$, and $X_n \to \infty$ a.s.
\end{enumerate}
In the former case, we call $(X_n)$ recurrent, and in the latter case, we call $(X_n)$ transient.
\end{lemma}

Similarly, a natural distinction between positive- and null-recurrence holds.

\begin{lemma}\label{Xn-positive-recur}
There exists a (unique) measure $\nu$ on $\ZP$ such that
\begin{equation}
\label{X-ergodic}
 \lim_{n \to \infty} \frac{1}{n} \sum_{k=0}^{n-1} \1 \{ X_k = x \} = \nu (x), \as ,\end{equation}
for any $x \in \ZP$. Exactly one of the following holds:
\begin{enumerate}
[label=(\roman{*})]
\item If $(X_n, \eta_n)$ is null, then $\nu (x) =0$ for all $x \in \ZP$. 
\item If $(X_n, \eta_n)$ is positive-recurrent, then $\nu(x) >0$ for all $x \in \ZP$ and $\sum_{x \in \ZP} \nu (x) =1$.
\end{enumerate}
If $(X_n)$ is recurrent, then we say that it is null-recurrent or positive-recurrent according
to which of (i) or (ii) holds.
\end{lemma}

The proofs of Lemmas \ref{Xn-recur-trans} and \ref{Xn-positive-recur} are standard and are omitted.

In the cases that
we consider, we will assume that the displacement of the $X$-coordinate has bounded $p$-moments for some $p < \infty$:
\begin{description}
\item
[\namedlabel{ass:X-diff-mom-p}{B$_\textit{p}$}]
There exists a constant $C_p < \infty$ such that $\Exp{[|X_{n+1}-X_n|^{p} \mid \F_n ]} \leq C_p, \as\ \forall n$.
\end{description}
In particular, \eqref{ass:X-diff-mom-p} for some $p>4$ will suffice for all of our results,
while for some of our results $p>1$ is sufficient.

Define $q_x(i,j) = \sum_{y \in \ZP} p(x,i,y,j)$.  We also assume:
\begin{description}
\item
[\namedlabel{ass:lim-q}{Q$_\infty$}]
$\lim_{x \to \infty} q_x(i,j) = q(i,j)$ exists for all $i,j \in S$, and $(q(i,j))$ is an irreducible stochastic matrix.
\end{description}

Note that since $\sum_{j \in S} q_x(i,j) = 1$, the limit in \eqref{ass:lim-q} is necessarily stochastic; however the irreducibility of $(q(i,j))$ does not follow from the irreducibility of $(q_x(i,j))$ for all $x \in \ZP$. For some of our results, it is necessary to assume a stronger condition than \eqref{ass:lim-q} that controls the rate of convergence of $q_x(i,j)$, namely:
\begin{description}
\item
[\namedlabel{ass:lim-q+}{Q$_\infty^\textrm{+}$}]
There exists $\delta_0 >0$ such that $ \max_{i,j \in S} |q_x (i,j) - q(i,j) | = O(x^{-\delta_0})$ as $x \to \infty$,
 and $(q(i,j))$ is an irreducible stochastic matrix.
\end{description}

Given \eqref{ass:lim-q}, we
 define $(\eta^\star_n)$ to be a Markov chain on $S$ with transition probabilities given by $q(i,j)$.  Since $(\eta^\star_n)$ is irreducible and finite
 there exists a unique stationary distribution $\pi$ on $S$ with $\pi (j) >0$ for all $j \in S$ and satisfying $\pi(j) = \sum_{i \in S} \pi(i) q(i,j)$.

\begin{remark}
\label{rem:homogeneity}
 A sufficient condition for \eqref{ass:lim-q+} is that there exists $x_0 \in \ZP$ such that 
\begin{description}
\item
[\namedlabel{ass:homogeneous_case}{H}] 
$p(x,i,y,j) = r(y-x , i , j)$
\end{description}
for all $x \geq x_0$,
i.e.,  for all $x$ large enough, the transition probabilities depend on $x$ and $y$ only through $y-x$. Then, $q_x ( i,j) = q(i,j) = \sum_{z \geq -x_0} r (z,i,j)$ for all $x \geq x_0$.
The homogeneity condition \eqref{ass:homogeneous_case} plays an important role in much of the existing literature, but is too restrictive for our purposes.
We discuss  \eqref{ass:homogeneous_case} and some of its consequences, including the connection to the theory of additive functionals
of Markov chains, in Section \ref{sec:additive_functionals} below. For now, we remark that if \eqref{ass:homogeneous_case} holds for all $x \geq x_0$,
then necessarily  $X_{n+1} - X_n$ is uniformly bounded below (by $-x_0$).
\end{remark}

We denote the moments of the displacements in the $X$-coordinate by
\[
\mu_k(x,i) := \Exp[ (X_{n+1}-X_n)^k \mid X_n = x, \eta_n = i] = \sum_{j \in S} \sum_{ y \in \ZP} (y-x)^k p(x,i,y,j) ; 
\]
then $\mu_1$ is well defined provided \eqref{ass:X-diff-mom-p} holds for some $p \geq 1$, while $\mu_2$ is finite if \eqref{ass:X-diff-mom-p} holds for some $p \geq 2$.
Our results will apply to the following two cases:
\begin{description}
\item
[\namedlabel{ass:mu1-const}{M$_\textrm{C}$}] There exist $d_i \in \R$ such that for all $i \in S$,
as $x \to \infty$,
 $\mu_1(x,i) = d_i + o(1)$;
\item
[\namedlabel{ass:mu-lamperti}{M$_\textrm{L}$}]
There exist $c_i \in \R$ and $s_i^2 \in \RP$, with at least one $s_i^2$ nonzero, such that for all $i \in S$,
as $x \to \infty$,
$\mu_1(x,i) = \frac{c_i}{x} + o(x^{-1})$ and $\mu_2(x,i) = s_i^2 +o(1)$.
\end{description}
Since $S$ is finite, the implicit constants in the $x \to \infty$ error terms in these expressions (and similar ones later on)
may be chosen uniformly over $i$.
Just as above, some of our results will require a stronger assumption than \eqref{ass:mu-lamperti} that controls the error terms as a function of $x$, namely:
\begin{description}
\item
[\namedlabel{ass:mu-lamperti+}{M$_\textrm{L}^\textrm{+}$}]
There exists $\delta_1 > 0$ such that, as $x \to \infty$,
\[
\mu_1(x,i) = \frac{c_i}{x} + O(x^{-1-\delta_1}) \text{ and } \mu_2(x,i) = s_i^2 + O(x^{-\delta_1}).
\]
\end{description}

Next we state our main results. The first two are concerned with the classification of the process as transient, null-recurrent, or positive-recurrent.
Of these, first we consider the case where each line is associated with a drift that is asymptotically constant, and where at least one of these
constants is nonzero.

\begin{theorem}
\label{thm:constant-drifts}
Suppose that  \eqref{ass:X-diff-mom-p} holds for some $p>1$, and conditions \eqref{ass:mu1-const} and \eqref{ass:lim-q} hold. Then the following classification applies.
\begin{itemize}
\item[(i)] If $\sum_{i \in S} d_i \pi (i) > 0$, then $X_n$ is transient.
\item[(ii)] If $\sum_{i \in S} d_i \pi (i) < 0$, then $X_n$ is positive-recurrent.
\end{itemize}
\end{theorem}

In the special case of \eqref{ass:lim-q} in which $q_x \equiv q$ does not depend on $x$,  Theorem \ref{thm:constant-drifts} is contained
in Theorem 3.1.2 of Fayolle \emph{et al.}~\cite{fmm}, who imposed, in part,
  an assumption
of a uniform lower bound on $X_{n+1} - X_n$. 
In the  generality of \eqref{ass:lim-q}, part (ii) is contained in a paper of Falin \cite{falin1}, who
also stated a version of part (i) assuming that \eqref{ass:homogeneous_case} holds for $x$ large enough.

The next result deals with the case of drift conditions of Lamperti-type. 

\begin{theorem}
\label{thm:lamperti_rough_then_sharp}
Suppose that  \eqref{ass:X-diff-mom-p} holds for some $p>2$, and conditions   \eqref{ass:lim-q} and \eqref{ass:mu-lamperti} hold.
The following sufficient conditions apply.
\begin{itemize}
\item If $\sum_{i \in S} (2c_i - s_i^2)\pi(i) > 0$, then $X_n$ is transient.
\item If  $| \sum_{i \in S} 2c_i \pi (i) | < \sum_{i \in S} s_i^2 \pi(i)$, then $X_n$ is null-recurrent.
\item If $\sum_{i \in S} (2c_i + s_i^2)\pi(i) < 0$, then $X_n$ is positive-recurrent.
\end{itemize}
If, in addition, \eqref{ass:lim-q+} and \eqref{ass:mu-lamperti+} hold, then the following condition also applies (yielding an exhaustive classification):
\begin{itemize}
\item If  $| \sum_{i \in S} 2c_i \pi (i) | = \sum_{i \in S} s_i^2 \pi(i)$, then $X_n$ is null-recurrent.
\end{itemize}
\end{theorem}

In the case where  $S$ is a singleton, Theorem \ref{thm:lamperti_rough_then_sharp}
reduces essentially to results of Lamperti \cite{lamp1,lamp3}, and so our result can be seen as a generalization
of Lamperti's.

Our final main result concerns the weak convergence of $(X_n, \eta_n)$.
The limit statement will involve the distribution function $F_{\alpha, \theta}$ defined
for  parameters $\alpha >0$ and $\theta > 0$ by
\begin{equation}
\label{eqn:Fdef}
 F_{\alpha, \theta} ( x) = \int_0^x \frac{2 u^{2\alpha -1} \re^{-u^2/\theta} }{ \theta^\alpha \Gamma (\alpha) } \ud u , ~~ (x \geq 0).\end{equation}
Note that, if $Z \sim \Gamma (\alpha, \theta)$ is a gamma random variable
with shape parameter $\alpha>0$ and scale parameter $\theta >0$, then $\Pr [ \sqrt{Z} \leq x ] = F_{\alpha, \theta} (x)$.
(In the special case with $\alpha = 1/2$ and $\theta =2$, $F_{\alpha, \theta}$ is 
the distribution of the square-root of a $\chi^2$ random variable with one degree of freedom, i.e.,
the absolute value of a standard normal random variable.)
 
\begin{theorem}
\label{thm:weak_limit}
Suppose that  \eqref{ass:X-diff-mom-p} holds for some $p>4$, and conditions   \eqref{ass:lim-q} and \eqref{ass:mu-lamperti} hold.
Suppose that the matrix $q$ appearing in \eqref{ass:lim-q} is aperiodic.
Suppose also that  
  $\sum_{i \in S} (2c_i + s_i^2)\pi(i) > 0$. Then, for any $k \in S$ and $x \in \RP$,
\[ \lim_{n \to \infty} \Pr \left[ n^{-1/2} X_n \leq x , \, \eta_n = k \right] = \pi (k) F_{\alpha, \theta} (x) ,\]
where
\begin{equation}
\label{eq:alpha_theta}
 \alpha = \frac{1}{2} + \frac{\sum_{i \in S} c_i \pi (i)}{\sum_{i \in S} s_i^2 \pi (i) } , ~~\text{and}~~
\theta = 2 \sum_{i \in S} s_i^2 \pi (i) .\end{equation}
\end{theorem}

 \begin{remarks} \label{rmk:weak_limit}
(i) 
 Under the hypothesis of Theorem \ref{thm:weak_limit}, Theorem \ref{thm:lamperti_rough_then_sharp} shows that the process
is null-recurrent or transient;  Theorem \ref{thm:weak_limit} demonstrates a form of asymptotic independence
between $X_n$ (rescaled) and $\eta_n$ (which converges to $\pi$). By contrast, in the positive-recurrent aperiodic case, 
$\Pr [   X_n \leq x , \, \eta_n = k  ]$ (with no scaling)  possesses a limit, but that limit
cannot be identified without additional assumptions (and the limit distribution of $\eta_n$ need not even be $\pi$). 

(ii) The case of Theorem \ref{thm:weak_limit}
in which $S$ is a singleton is essentially Lamperti's weak convergence result from \cite{lamp2}.

(iii)
If in addition \eqref{ass:lim-q+} and \eqref{ass:mu-lamperti+} hold, then the boundary case $\sum_{i \in S} (2c_i + s_i^2)\pi(i) = 0$
is null-recurrent, by Theorem \ref{thm:lamperti_rough_then_sharp}. In this case the proof given in Section \ref{sec:weak_convergence} below can be modified to show that
$n^{-1/2} X_n \to 0$ in probability; this is consistent with the fact that the $\alpha \to 0$ limit of $F_{\alpha, \theta}$ corresponds to a point mass at $0$.

(iv) 
With some additional work, the arguments in Section \ref{sec:weak_convergence} should yield the process version of Theorem \ref{thm:weak_limit}:
in the sense of finite dimensional distributions, as $n \to \infty$,
\[ \left( n^{-1/2} X_{nt} , \eta_{nt}  \right)_{t \in [0,1]} \longrightarrow \left( x_t, \omega_t \right)_{t \in [0,1]} ,\]
 where $(2/\theta)^{1/2} x_t$ is a Bessel process with dimension $2\alpha$ and $\omega_t$ is an $S$-valued white noise
process whose finite-dimensional marginals are sequences of i.i.d.~$\pi$-distributed variables.
\end{remarks} 

The remainder of the paper is organized as follows. In Section \ref{sec:examples} we 
give some additional context to the present work by describing how our setting
generalizes the literature on additive functionals of Markov chains, and by presenting
some additional examples, including a variant of the correlated random walk. Section \ref{sec:analysis}
contains the bulk of our analysis, which proceeds via considering an embedded Markov chain. The proofs of the
main theorems are then completed in Section \ref{sec:proofs}.

To simplify the presentation in the rest of the paper, we often write $\Pr_{x, i} [ \, \cdot \, ]$ for $\Pr [ \, \cdot \, \mid X_0 =x, \eta_0 =i]$,
corresponding to the law of the Markov chain with initial state $(x,i) \in \ZP \times S$; similarly for (expectation) $\Exp_{x,i}$.

We finish this section with some general remarks. 
Our method of proof is different from other approaches in the literature. Falin \cite{falin1,falin2}, while also making use of Foster--Lyapunov
results, bases his computations on a delicate algebraic calculation. Rogers \cite{rogers} uses an embedded Markov chain, as we do, but his analysis relies
on the additive functional representation (see Section \ref{sec:additive_functionals}). 
Our approach to the excursion estimates for the embedded process, via the Doob decomposition, makes the emergence of the `pseudo-drift'
quantities particularly intuitive from a probabilistic perspective: see the discussion around \eqref{eqn:excursion_drift} below.

The case where $S$ is infinite can give rise to completely different phenomena from the finite setting, and we do not consider this here. Under suitable assumptions, however,
such as uniform versions of our asymptotic conditions \eqref{ass:lim-q}, \eqref{ass:mu1-const} or \eqref{ass:mu-lamperti},
and sufficient moments for $\tau$ and the increments of $X_n$, the results of the present paper should extend to the infinite setting.

\section{Examples and remarks on the literature}
\label{sec:examples}

\subsection{Homogeneity and additive functionals}
\label{sec:additive_functionals}

As mentioned in Remark \ref{rem:homogeneity}, condition \eqref{ass:homogeneous_case}
is assumed in much of the literature. A special structure emerges
when 
\eqref{ass:homogeneous_case} is imposed for \emph{all} $x$.
Indeed, one then has
that $(\eta_n)$ itself is a Markov chain, since
\begin{align}
\label{eqn:homogeneous_q}
 \Pr [ \eta_{n+1} = j  \mid (X_n,\eta_n) = ( x , i ) ] = \sum_{y \in \Z} r ( y -x , i ,j )  = \sum_{z \in \Z} r ( z , i ,j ) = q (i,j) .\end{align}
A similar argument shows that $(X_{n} - X_{n-1} , \eta_{n} )$ is a Markov chain on $\Z \times S$, with
\begin{align*}
 \Pr [ (X_{n+1}-X_{n}, \eta_{n+1} ) = (z,  j)  \mid (X_{n} - X_{n-1} , \eta_{n} ) = ( y , i ) ] 
= r (z, i ,j).\end{align*}
Then if $\psi : \Z \times S \to \Z$ is given by $\psi (z,i) = z$, we may write
\[ X_n = X_0 + \sum_{k=0}^{n-1} \psi ( X_{k+1} - X_k , \eta_{k+1} ),\]
which represents $X_n$ as an \emph{additive functional of a Markov chain}. 

However, for $x \in \ZP$, assuming that \eqref{ass:homogeneous_case} holds for all $x \geq 0$
is very restrictive, and implies that $X_{n+1} - X_n \geq 0$ a.s.~(see Remark \ref{rem:homogeneity}).
So in the homogeneous setting, it makes sense to  instead take the state space to be $\Z \times S$
so that \eqref{eqn:transition_probs}
now holds with $x$ and $y$ in $\Z$. Assuming that \eqref{ass:homogeneous_case} holds for all $x \in \Z$
now yields the additive functional structure above, without imposing additional restrictions on the magnitude of
$X_{n+1} - X_n$.

In either case, we may note 
 that 
\begin{equation}
\label{eqn:homo_drifts}
 \Exp [ X_{n+1} - X_n \mid (X_n , \eta_n ) = (x,i) ] = \sum_{z \in \Z} \sum_{j \in S} z r (z, i ,j) =: \mu_1 (i) ,\end{equation}
say,
assuming that the mean increments are well defined; so there is a constant mean drift $\mu_1 (i)$ for each $i \in S$.

Moreover, if $\pi$ is the stationary distribution on $S$ associated with the Markov chain $(\eta_n)$
given by \eqref{eqn:homogeneous_q}, then a calculation shows that the Markov chain 
$(X_{n} -X_{n-1} , \eta_n)$ has stationary distribution $\varpi(z,i)$ on $\Z \times S$ given by
\[
 \varpi (z, i)  =  \sum_{k \in S} \pi (k) r ( z , k,i) .\]

In this context, a result of Rogers \cite{rogers}
on additive functionals of Markov chains shows that recurrence classification
of $(X_n)$ 
depends on the sign of
\begin{align} 
\label{eqn:excursion_drift}
\sum_{i \in S} \sum_{z \in \Z} \varpi (z,i) \psi (z, i) 
= \sum_{i \in S} \sum_{z \in \Z} z \sum_{k \in S} \pi (k) r ( z , k,i) 
 = \sum_{k \in S} \pi(k) \mu_1 (k) .\end{align}
There are many similar results in the literature for additive functionals of Markov chains in more general spaces, 
and related results in ergodic theory concerning `co-cycles' (see, e.g.,~\cite{atkinson}).
However, the methods adapted to this additive functional structure seem to depend
 crucially on the homogeneity assumption \eqref{ass:homogeneous_case}.

The interpretation of the quantity of \eqref{eqn:excursion_drift}
is as a `pseudo-drift' accumulated over i.i.d.\ excursions of the Markov chain:
see Rogers \cite{rogers}. We take this idea further, as the analogues of these excursions in our setting are not i.i.d.,
due to the additional non-homogeneity. However, our methods exploit the essential structure that remains.

\subsection{Correlated random walk}
\label{sec:correlated_walk}

In the one-dimensional {\em correlated random walk}, a particle performs a random walk on $\Z$ with a short-term memory: the distribution of $X_{n+1}$
depends not only on the current position $X_n$, but also on the `direction of travel' $X_n - X_{n-1}$. Formally,
$(X_n , X_n - X_{n-1} )$ is a Markov chain on $\Z \times \{ -1 , +1 \}$.
Supposing also that \eqref{ass:homogeneous_case} holds for all $x \in \Z$, this is a special case of the framework
discussed in Section \ref{sec:additive_functionals}, with $\eta_n = X_n - X_{n-1}$.

One standard version of the model
supposes that
the nonzero transition probabilities are given by
$p (x, i, x+j , j) = r(j, i, j) = q(i,j)$, where
\[ q(i,j) = \begin{cases} 
\frac{1}{2} + \rho_i & \text{if } j =i \\
\frac{1}{2} - \rho_i & \text{if } j \neq i
\end{cases}
\]
is  the transition matrix of the Markov chain  $(\eta_n)$, and $\rho_i \in (-\frac{1}{2}, \frac{1}{2})$ are fixed
parameters. For this random walk, the additive
structure described
in Section \ref{sec:additive_functionals} is particularly simply expressed via $X_n = X_0 + \sum_{k=0}^{n-1} \eta_{k+1}$. 

Corresponding to $q$ is the stationary distribution 
$\pi(i) = \frac{(1/2) - \rho_{-i}}{1-\rho_i-\rho_{-i}}$,
and the mean drifts given by \eqref{eqn:homo_drifts} are now
$\mu_1 (i) = \sum_{j \in S} j q (i, j) = 2 i \rho_i$.
Then we see that the
`pseudo-drift' \eqref{eqn:excursion_drift} is zero 
 if and only if $\rho_i = 
\rho$ is the same for each $i$; the 
random walk is recurrent
in exactly this case.

A positive $\rho_i$ corresponds to \emph{persistence}
of the walker in direction $i$ (the walker has an `inertia');
a negative $\rho_i$ corresponds to a walker who \emph{vacillates} in direction $i$,
and has an increased propensity to turn around.
 
Such models have a long history, and have been studied under different names by many different researchers:
as `persistent random walks' by F\"urth \cite{furth}, `correlated random walks' by Gillis  \cite{gillis},
`random walks with restricted reversals' by  Domb and Fisher \cite{df},
and, recently, `Newtonian random walks' by Lenci \cite{lenci}.
 Under appropriate rescaling, the  model leads to the {\em telegrapher's equation} in the scaling limit, as discussed by Goldstein \cite{goldstein} and Kac 
\cite{kac}. 
There has been a large amount of recent work on correlated random walk and related models;
a small selection is \cite{am1,st,cr2,hs}. Motivation for studying these models arises from several
sources, including physical Brownain motion \cite{furth} and models for molecular configurations \cite{daniels}. We refer
to \cite{hughes} for some additional background and references.
 
As an application of our main results, consider the following variation on
 the one-dimensional correlated random walk, intended to probe more
precisely the recurrence-transience phase transition. This time
we take the state-space to be $\ZP \times \{ -1, +1\}$ to fit into the setting
of Section \ref{sec:model}. We
suppose that 
the nonzero transition probabilities are
$p (x, i, x+j , j) = q_x(i,j)$, where
\[ q_x (i,j) = \begin{cases} 
\frac{1}{2} + \frac{i c }{2x} + O (x^{-1-\delta} )& \text{if } j =i \\
\frac{1}{2} - \frac{i c }{2x} + O (x^{-1-\delta} )& \text{if } j \neq i
\end{cases}
\]
for some constants  $\delta >0$ 
 and $c  \in \R$. For $c>0$, the walk is persistent in the positive direction but vacillating
in the negative direction; conversely for $c <0$. So for nonzero $c$, the symmetry between the two directions
present in the (recurrent) $c=0$ case
is broken: how does this affect the recurrence?

Under these assumptions, 
\eqref{ass:lim-q+} holds with
$q (i, j) = \frac{1}{2}$ for all $i, j$, so that $\pi(i) = \frac{1}{2}$
for $i = \pm 1$. Also,
\[ \mu_1 (x, i) = \sum_{j \in S} j q_x (i, j) = \frac{c }{x}  + O (x^{-1-\delta} ) , ~~\text{and}~~ \mu_2 (x, i) = \sum_{j \in S} j^2 q_x (i,j) = 1, \text{ for $x \geq 1$}
,\]
so that \eqref{ass:mu-lamperti+} holds (with $s_i^2 = 1$ for $i =\pm1$). Applying Theorems \ref{thm:lamperti_rough_then_sharp}
and \ref{thm:weak_limit}
yields the following result.

\begin{corollary}\label{cor:corr-rw}
If $c < -\frac{1}{2}$, then the walk is positive-recurrent. If $c > \frac{1}{2}$, then the walk is transient.
If $| c | \leq \frac{1}{2}$, then the walk is null-recurrent.
Moreover, if $c > - \frac{1}{2}$, then 
\[ \lim_{n \to \infty} \Pr [ n^{-1/2} X_n \leq x , \, \eta_n = i ] = \frac{1}{2} F_{c+(1/2) , 2} (x).\] 
\end{corollary}

\subsection{Modulated queue}
\label{sec:modulate-queue}

To finish this section we return to the queueing model as presented in the introduction. 
 Recall that the critical case from the point of view of recurrence and transience is when $\rho = \lambda$, and we are interested in the behaviour of the model under perturbations of the constants $c_i$ for $i \in S$.  For this model we have $q_x(i,j) = \frac{1}{2}(a_{ij} + b_{ij}) + O(x^{-1})$ so provided that the matrix $M_{ij} = \frac{1}{2}(a_{ij}+b_{ij})$ is irreducible, 
condition \eqref{ass:lim-q+} holds.  We see that
\[
\mu_1(x,i) = \frac{c_i}{x} + O(x^{-2}), \quad \text{and} \quad \mu_2(x,i) = 1,
\]
so that \eqref{ass:mu-lamperti+} holds.  Let $\pi$ be the stationary distribution associated with 
transition matrix~$M$, and set $\bar{c} = \sum_{i \in S} c_i \pi(i)$. Applying Theorems \ref{thm:lamperti_rough_then_sharp}
and \ref{thm:weak_limit}
yields the following result (cf.\ Corollary~\ref{cor:corr-rw}).

\begin{corollary}
 If $\bar{c} < -\frac{1}{2}$, then the Markov chain is positive-recurrent. If $\bar{c} > \frac{1}{2}$, then the Markov chain is transient.
If $| \bar{c} | \leq \frac{1}{2}$, then the Markov chain is null-recurrent.
Moreover, if $\bar{c} > - \frac{1}{2}$, then 
\[ \lim_{n \to \infty} \Pr [ n^{-1/2} X_n \leq x , \, \eta_n = i ] = \frac{1}{2} F_{\bar{c}+(1/2) , 2} (x).\] 
\end{corollary}

\section{Analysis via an embedded Markov chain}
\label{sec:analysis}

\subsection{Overview}
\label{sec:overview}

To analyse $(X_n,\eta_n)$ we look at an embedded process $(Y_n)$, which records the $X$-coordinate of the chain when it returns to a
given line. 
Formally, we label an arbitrary state $0 \in S$. Then
set $\tau_0 = \min\{n \in \ZP : \eta_n = 0\}$, and for $m \geq 0$ set $\tau_{m+1} = \min\{ n > \tau_m : \eta_n = 0\}$, where 
we adopt the usual convention that $\min \emptyset = \infty$.  To ease exposition, we introduce a `coffin' state $\partial$ and define the embedded process $Y_n$ on $\ZP \cup \{ \partial \}$ by 
\[
Y_n = \begin{cases} X_{\tau_n} & \text{if $\tau_n < \infty$,}\\ \partial & \text{if $\tau_n = \infty$.} \end{cases}
\]
We also introduce $\tau = \min\{ n > 0 : \eta_n = 0 \}$ 
(so $\tau = \tau_0\1{\{\eta_0 \neq 0\}} + \tau_1 \1{\{\eta_0 = 0\}}$).

For any $n \in \ZP$, given $\tau_n < \infty$ and $X_{\tau_n} = x$, the strong Markov property for the time-homogeneous Markov chain $(X_n,\eta_n)$ shows that $(X_{\tau_n+m}, \eta_{\tau_n+m})_{m\geq 0}$ is independent of $(X_0,\eta_0),\dotsc,(X_{\tau_n},\eta_{\tau_n})$ and is distributed as a copy of $(X_m,\eta_m)_{m \geq 0}$ given $(X_0,\eta_0) = (x,0)$.   In particular, on $\tau_{n+1} < \infty$, the pair $(X_{\tau_{n+1}},\tau_{n+1} - \tau_n)$ depends on $(X_0,\eta_0),\dotsc,(X_{\tau_n},\eta_{\tau_n})$ only through $X_{\tau_n}$.  Hence $Y_n$ is a Markov chain and, given $Y_n = x$, the random variable $\tau_{n+1} - \tau_n$ has the same distribution as $\tau$ conditional on $(X_0,\eta_0)= (x,0)$.

We refer to $(X_m, \eta_m)_{\tau_n \leq m \leq \tau_{n+1}}$ as the $n$th \emph{excursion} from the line $0$.
The basis for our analysis of the embedded Markov chain $(Y_n)$ will be an analysis
of a single excursion, depending on the starting position. A key component of this analysis is a coupling result,
which we present in the next subsection.

\subsection{Coupling construction}

\begin{lemma}\label{lem:coupling}
Suppose that condition \eqref{ass:X-diff-mom-p} holds for some $p>1$ and condition \eqref{ass:lim-q} holds.  Then there exists a Markov chain $(X_n,\eta_n,\eta^\star_n)$ on $\ZP \times S \times S$ such that
\begin{itemize}
\item $(X_n,\eta_n)$  
is a Markov chain on $\ZP \times S$ with transition probabilities $p(x,i,y,j)$;
\item $(\eta^\star_n)$ is a Markov chain on $S$ with transition probabilities $q(i,j)$; and
\item for all $n \in \ZP$ and all $i \in S$,
\begin{equation}\label{eqn:eta-eta-star-coupling}
\lim_{x\to \infty} \Pr{ \bigg[ \bigcap_{0 \leq k \leq n} \{ \eta_k = \eta^\star_k \} \bigmid X_0 = x, \eta_0 = \eta^\star_0 = i \bigg]} = 1.
\end{equation}
\end{itemize}
Finally, suppose in addition that \eqref{ass:lim-q+} holds. Then there exists $\delta>0$ such that, for any $A < \infty$, for all $i \in S$,
as $x \to \infty$,
\begin{equation}
\label{eqn:eta-eta-star-coupling2}
1 -  \Pr{ \bigg[ \bigcap_{0 \leq k \leq A \log x} \{ \eta_k = \eta^\star_k \} \bigmid X_0 = x, \eta_0 = \eta^\star_0 = i \bigg]} = O (x^{-\delta} ) .
\end{equation}
\end{lemma}

The statements of Lemma~\ref{lem:coupling} will follow from a
coupling argument.  Essentially, equation~\eqref{eqn:eta-eta-star-coupling} is proved using a maximal coupling of $\eta_n$ and $\eta^\star_n$; the condition~\eqref{ass:lim-q} that $q_x(i,j)$ has a limit as $x \to \infty$ means that we can control the probability of decoupling, \emph{provided} that $X_n$ stays sufficiently large, and it is this dependence on $X_n$ that introduces a (minor) complication to an otherwise standard argument.  Equation~\eqref{eqn:eta-eta-star-coupling2} is proved in a similar manner using the stronger condition~\eqref{ass:lim-q+} on $q_x(i,j)$; the full details of the proof can be found in Appendix~\ref{sec:technical_appendix}.

In the remainder of this subsection we explore some consequences of the coupling described in Lemma \ref{lem:coupling}.
First we introduce additional notation in the context of the joint probability space on which the coupled
process $(X_n, \eta_n, \eta^\star_n)$ is constructed.
 We denote by $\tau^\star$ the first return time to 0 of the Markov chain $(\eta^\star_n)$, namely
\[
\tau^\star := \min \{ n \geq 1 : \eta^\star_n = 0 \}.
\]
Moreover, we write $\Pr_{x,i,j}$ for the
probability measure conditional on $X_0 =x, \eta_0 = i, \eta^\star_0 = j$,
and $\Exp_{x,i,j}$ for the corresponding expectation.

Irreducibility of the time-homogeneous Markov chain $(X_n,\eta_n)$ and finiteness of $S$ imply that for any $x$, there exist $m(x) < \infty$ and $\varphi(x)>0$ such that
\begin{equation}\label{eqn-irred-at-x}
\Pr_{x,i} { [\tau \leq m(x) ]} \geq \varphi(x) \text{ for all $i$}.
\end{equation}
In the specific case that $q_x(i,j)$ is constant in $x$, the process $(\eta_n)$ is distributed exactly as the finite
irreducible Markov chain $(\eta^\star_n)$, so 
the functions $m(x)$ and $\varphi(x)$ in \eqref{eqn-irred-at-x} can be chosen to be
 \emph{uniform} over $x$.
Our first consequence of the above coupling 
is that \eqref{eqn-irred-at-x}
can be strengthened to such a uniform version  under our weaker conditions: roughly speaking, assumption \eqref{ass:lim-q} implies that $(\eta_n)$ is sufficiently close to  $(\eta^\star_n)$ 
 when the $X$-coordinate of $(X_n,\eta_n)$ is sufficiently large, and irreducibility does the rest.

\begin{lemma}
\label{lem:q-implies-unif-irred}
Suppose that condition \eqref{ass:X-diff-mom-p} holds for some $p>1$ and condition \eqref{ass:lim-q} holds.  Then there exist $m < \infty$ and $\varphi > 0$ such that,
 for all $i$  and all $x$,
\begin{equation}\label{eqn-unif-irred}
\Pr_{x,i} {[\tau \leq m ]} \geq \varphi.
\end{equation}
\end{lemma}

In the proof of this result, and at several points later on, we consider the event
\begin{equation}
\label{eq:E_n_def}
 E_n := \cap_{0 \leq \ell \leq n} \{ \eta_\ell = \eta^\star_\ell  \} .\end{equation}

\begin{proof}[Proof of Lemma \ref{lem:q-implies-unif-irred}.]
We work with the Markov chain $(X_n, \eta_n, \eta^\star_n)$ given in Lemma \ref{lem:coupling}.
Since $\eta^\star$ is a finite
 irreducible Markov chain, there exist $m < \infty$ and $\varphi > 0$ such that
$\Pr_{x,i,i}[ \tau^\star \leq m ] \geq 2\varphi$
for all $i$ and all $x$.  Conditional on $\eta_n$ and $\eta^\star_n$ remaining coupled up to time $m$, we have $\tau \leq m$ if and only if $\tau^\star \leq m$; hence
\[
\Pr_{x,i,i}[ \tau \leq m ] \geq \Pr_{x,i,i}[E_m \cap \{ \tau^\star \leq m\}] \geq \Pr_{x,i,i}[\tau^\star \leq m] - \Pr_{x,i,i}[E_m^\rc].
\]
But by Lemma \ref{lem:coupling}, there exists $x_0$ such that $\Pr_{x,i,i}[E_m^\rc] \leq \varphi$ for all $x \geq x_0$ and hence  
\eqref{eqn-unif-irred} holds for all $i$ and all $x \geq x_0$.
 
  But also, since  $x_0 < \infty$, for any $m(x)$ and $\varphi(x)$
  satisfying \eqref{eqn-irred-at-x}, we define $m_0 := \max_{x \leq x_0} m(x) < \infty$
  and $\varphi_0 := \min_{x \leq x_0} \varphi (x) >0$ so that, for any $x \leq x_0$,
\[
 \Pr_{x, i} {[  \tau \leq m_0  ]} \geq \Pr_{x,i} {[\tau \leq m(x) ]} \geq \varphi(x) \geq \varphi_0, \text{ for all }  i.
\]
So, redefining $m$ and $\varphi$ as necessary,  \eqref{eqn-unif-irred} in fact holds for all $i$ and \emph{all} $x$.
\end{proof}

\subsection{Excursion durations and occupation estimates}

Next we give an exponential tail bound for the duration of excursions, uniform in the initial location.
 
\begin{lemma}\label{lem:unif-irred}
Suppose that condition \eqref{ass:X-diff-mom-p} holds for some $p>1$ and condition \eqref{ass:lim-q} holds.
Then there exist constants $c >0$ and $C< \infty$ such that,  for all $x$, $n$, and $r$,
\[
\Pr{[\tau_{n+1} - \tau_n > r \mid X_{\tau_n} = x ]} \leq C \re^{-c r}.
\]
\end{lemma}
\begin{proof}
Recall that since $\tau_{n+1} - \tau_n$ conditional on $Y_n = x$ has the same distribution as $\tau$ conditional on $X_0 = x, \eta_0 = 0$, it suffices to show that, for some constants $C,c > 0$,
\begin{equation}\label{eqn:tau-exp-tails}
\Pr_{x,i} {[ \tau > r ]} \leq C \re^{-cr}, ~\text{for all $x$ and $i$}.
\end{equation}
 (We then get the claimed result for $\tau_{n+1} - \tau_n$ by setting $i=0$.) 
Recall that, by Lemma~\ref{lem:q-implies-unif-irred},
$\Pr_{x,i} {[  \tau \leq m ]} \geq \varphi$.
 Moreover, using the time-homogeneity of $(X_n,\eta_n)$,   for all $x$ and $i$,
\[
\begin{split}
\Pr_{x,i} {[ \tau \leq km+m \mid \tau> km ]} & \geq \min_{y,j} \Pr{[\tau \leq km+m \mid \tau > km, X_{km} = y, \eta_{km} = j]}\\
&= \min_{y,j} \Pr{[\tau \leq m \mid X_0 = y, \eta_0 = j]} \geq \varphi,
\end{split}
\]
for all positive integers $k$.  But this implies that, for all positive integers $k$,
\[
\Pr_{x,i} {[ \tau > km  ]} = \prod_{j=1}^k \Pr_{x,i} {[\tau > jm \mid \tau > (j-1)m ]} \leq (1-\varphi)^k.
\]
Finally, for general $r \in \ZP$, there exists an integer $k$ such that $km \leq r < (k+1)m$, so
\begin{align*}
\Pr_{x,i} {[ \tau > r  ]} & \leq \Pr_{x,i} {[ \tau > km  ]}   \leq (1-\varphi)^k < (1-\varphi)^{r/m-1} \leq C\re^{-cr},
\end{align*}
for constants $C,c > 0$ dependent only on $\varphi$ and $m$, giving~\eqref{eqn:tau-exp-tails}.
\end{proof}

 The next result   shows that the mean  occupation time of $(X_n,\eta_n)$ on line $i$ per excursion
 can be approximated by the mean  occupation time of $(\eta^\star_n)$ in state $i$ per excursion.

\begin{lemma}
\label{lem:occupation-limit}
Suppose that condition \eqref{ass:X-diff-mom-p} holds for some $p>1$ and condition \eqref{ass:lim-q} holds. Then, 
for  any $i \in S$,
\[ \lim_{x \to \infty} \Exp_{x,0}   \sum_{k=0}^{\tau -1} \1 \{ \eta_k = i \}  = \frac{\pi(i)}{\pi(0)} .\]
If, in addition, \eqref{ass:lim-q+} holds, then there exists $\delta>0$ such that, for  any $i \in S$, as $x \to \infty$,
\[ \bigg| { \Exp_{x,0}   \sum_{k=0}^{\tau -1} \1 \{ \eta_k = i \}  - \frac{\pi(i)}{\pi(0)} } \bigg| = O ( x^{-\delta}).\]
\end{lemma}
\begin{proof}
Again we work with the Markov chain $(X_n, \eta_n, \eta^\star_n)$ whose existence is
given in the statement of Lemma~\ref{lem:coupling}. 
Fix $i \in S$. 
For the duration of this proof, we write
\[ W:= \sum_{k=0}^{\tau -1} \1 \{ \eta_k = i \}, \text{ and } W^\star := \sum_{k=0}^{\tau^\star -1} \1 \{ \eta_k^\star = i \} .\]
Since $(\eta^\star_n)$ is a Markov chain on $S$
with transition probabilities $q(i,j)$, standard Markov chain theory
yields
$\Exp_{x,0,0} [ W^\star  ] =  \pi(i)/\pi(0)$, 
for any $x \in \ZP$. The statements of the lemma will follow
from suitable estimates for $\Exp_{x,0,0} [ | W - W^\star | ]$. 

Again define $E_n$ by \eqref{eq:E_n_def}.
Then, for any positive integer $n$,
\begin{align*}
\Exp_{x,0,0} \left[ | W- W^\star |  \right] & \leq \Exp_{x,0,0} \left[ | W-W^\star | \1 ( E_n ) \1 \{ \tau \vee \tau^\star \leq n \} \right]\\
& {} \quad {} + \Exp_{x,0,0} \left[ | W- W^\star| \1 ( E_n^\rc ) \1 \{  \tau \vee \tau^\star \leq n \} \right] \\
& {} \quad {} + \Exp_{x,0,0} \left[ | W- W^\star| \1 \{  \tau \vee \tau^\star > n \} \right] \\
& \leq 0 + n \Pr_{x,0,0} \left[ E_n^\rc   \right] + \Exp_{x,0,0} \left[ (\tau \vee \tau^\star) \1 \{  \tau \vee \tau^\star > n \} \right]
.\end{align*}
Moreover,
\begin{align}
\label{e:tautau}
 \Exp_{x,0,0} \left[ (\tau \vee \tau^\star) \1 \{  \tau \vee \tau^\star > n \} \right]
& \leq \Exp_{x,0,0} \left[ \tau  \1 \{  \tau > n \} \right] + \Exp_{x,0,0} \left[ \tau^\star  \1 \{  \tau^\star > n \} \right] . \end{align}
Here, by Cauchy--Schwarz and the tail estimates in Lemma~\ref{lem:unif-irred},
\begin{equation}
\label{tau_upper}
 \Exp_{x,0,0} [ \tau  \1 \{  \tau > n \} ] \leq ( \Exp_{x,0,0} [ \tau^2 ] )^{1/2} ( \Pr_{x,0,0} [ \tau > n ] )^{1/2}
\leq C \re^{-c n} ,\end{equation}
for some constants $C < \infty$ and $c>0$, not depending on $x$, and similarly for the term involving $\tau^\star$.
 For the first statement in the lemma,
it suffices to show that
\begin{equation}
\label{eqn:occupation-convergence}
 \lim_{x \to \infty} \Exp_{x,0,0} \left[ \left| W - W^\star \right|   \right]  = 0.\end{equation}
Under assumption \eqref{ass:lim-q}, it follows from \eqref{tau_upper} and its analogue for $\tau^\star$
that for any $\eps >0$ we may choose $n \geq n_0$ sufficiently large so that the right-hand side of \eqref{e:tautau}
 is less than~$\eps$, and then $\Exp_{x,0,0} [ | W- W^\star |  ] \leq n  \Pr_{x,0,0} [ E_n^\rc   ] + \eps$. 
For fixed $n$, $\Pr_{x,0,0} [ E_n^\rc   ] \to 0$ as $x \to \infty$ by \eqref{eqn:eta-eta-star-coupling},
so that $\limsup_{x \to \infty}  \Exp_{x,0,0} [ | W- W^\star |  ] \leq \eps$.  Since $\eps >0$ was arbitrary, 
 \eqref{eqn:occupation-convergence} follows.

For the second statement in the lemma, under assumption  \eqref{ass:lim-q+}, we use a similar argument
but with $n = n(x) = \lfloor A \log x \rfloor$. As before,
\[ \Exp_{x,0,0} [ | W- W^\star |  ] \leq n(x) \Pr_{x,0,0} [ E_{n(x)}^\rc   ] + \Exp_{x,0,0} [ (\tau \vee \tau^\star) \1 \{  \tau \vee \tau^\star > n(x) \} ] .\]
For a sufficiently large choice of constant $A$, the
exponential
bound \eqref{tau_upper} shows that the right-hand side of \eqref{e:tautau} decays as a power of $x$, for $n = n(x)$.
Finally, the term $n(x) \Pr_{x,0,0} [ E^\rc_{n(x)} ]$ also decays as a power of $x$, by \eqref{eqn:eta-eta-star-coupling2},
and so we see that $\Exp_{x,0,0} [ | W- W^\star |  ]$ decays as a power of $x$, as required.
\end{proof}

\subsection{Recurrence and transience relationships}
\label{sec:recurrence-transience}

In this subsection we demonstrate the equivalence of recurrence properties of the embedded process $(Y_n)$ to those of
the process $(X_n)$.
 
 From this point of the paper onwards, we will be increasingly concerned with multiple excursions, and it is useful
to introduce the 
notation $\sigma_0 := 0$ and, for $n \in \ZP$,
\[ \sigma_{n+1} := \tau_{n+1} - \tau_n \]
for the durations of the excursions. 
Recall the definition of $Y_n$ from Section \ref{sec:overview}.
Under our conditions (cf.~Lemma \ref{lem:unif-irred}),
$\sigma_n < \infty$ a.s.~for each $n$. Hence $Y_n \neq \partial \as$, and
we can identify $Y_n$ with $X_{\tau_n}$ for all $n$.
For the remainder of the paper
we employ this slight abuse of notation, and assume that the state space of $(Y_n)$ is $\ZP$.
The next result relates recurrence of $(X_n)$ to recurrence of $(Y_n)$.

\begin{lemma}
\label{lem:X-Y}
Suppose that condition \eqref{ass:X-diff-mom-p} holds for some $p>1$ and condition \eqref{ass:lim-q} holds.      
Then the process $(Y_n)$ is an irreducible Markov chain (on $\ZP$). Moreover
\begin{enumerate}[label=(\roman{*})]
\item $(X_n)$ is recurrent if and only if $(Y_n)$ is recurrent.
\item $(X_n)$ is positive-recurrent if and only if $(Y_n)$ is positive-recurrent.
\end{enumerate}
\end{lemma}
\begin{proof}
As explained in Section \ref{sec:overview}, the fact that $(Y_n)$ is a Markov chain follows from the strong Markov property for $(X_n,\eta_n)$.

Irreducibility of $(Y_n)$ follows from the irreducibility of $(X_n,\eta_n)$, as follows.  For any $x,y \in \ZP$, there exists a finite path in the state space $\ZP \times S$ from $(x,0)$ to $(y,0)$ that the chain $(X_n,\eta_n)$ has a positive probability of following.  But then the (finite) subpath consisting of the points that are on line 0 corresponds to a path in the state space $\ZP$ that $(Y_n)$ has a positive probability of following.

Now, for statement (i), the fact that $Y_n = 0$ exactly when $X_{\tau_n} = 0$ implies  
 $\{Y_n = 0 \io\}$ if and only if $\{(X_n,\eta_n) = (0,0) \io\}$, so $(Y_n)$ is recurrent if and only if $(X_n,\eta_n)$ is recurrent.  Using Lemma~\ref{Xn-recur-trans}, we have $(Y_n)$ is recurrent if and only if $(X_n)$ is recurrent.

Finally, we verify (ii). Let 
\[ \xi = \min \{ n \geq 1 : Y_n = 0 \}, \text{ and } \zeta = \min \{ n \geq 1 : (X_n, \eta_n) = (0,0) \} .\]
Then $(Y_n)$ is positive-recurrent if and only if $\Exp_{x,0} \xi < \infty$ for some (hence all) $x$, while $(X_n, \eta_n)$
is positive-recurrent if and only if $\Exp_{x,0} \zeta < \infty$. However, $\xi$ and $\zeta$ are related since, given $\eta_0 = 0$, it is the case that
$\tau_0 = 0$ and
$\zeta = \tau_\xi$, i.e.,
\begin{equation}
\label{e:xi-zeta}
 \zeta = \sum_{k=0}^{\xi -1} \sigma_{k+1}  = \sum_{k=0}^\infty \sigma_{k+1} \1 \{ k < \xi \} .\end{equation}
In particular, \eqref{e:xi-zeta} shows that $\zeta \geq \xi$, a.s., so $\Exp_{x,0} \zeta < \infty$ implies that $\Exp_{x,0} \xi < \infty$.
For the implication in the other direction, take expectations in the final expression in \eqref{e:xi-zeta} and
use linearity of expectations and Fubini's Theorem to get
\begin{align*}
\Exp_{x,0} \zeta & = \Exp_{x,0} \sum_{k=0}^\infty \Exp \left[ \sigma_{k+1} \1 \{ k < \xi \} \mid \F_{\tau_k} \right] \\
& = \Exp_{x,0} \sum_{k=0}^\infty \1 \{ k < \xi \}  \Exp \left[  \sigma_{k+1}   \mid \F_{\tau_k} \right] ,
\end{align*}
since $\{ k < \xi \} \in \F_{\tau_k}$. But, by Lemma \ref{lem:unif-irred}, 
$\Exp \left[  \sigma_{k+1} \mid \F_{\tau_k} \right]$
is uniformly bounded by a constant, $C$, say, so that
\[ \Exp_{x,0} \zeta \leq C \Exp_{x,0} \sum_{k=0}^\infty \1 \{ k < \xi \} = C \Exp_{x,0} \xi .\]
Hence $\Exp_{x,0} \zeta < \infty$ if and only if $\Exp_{x,0} \xi < \infty$.
Finally, (ii) follows from Lemma \ref{Xn-positive-recur}, 
which gives the equivalence of positive-recurrence for $(X_n, \eta_n)$ and $(X_n)$.
\end{proof}

\subsection{Increment moment estimates}
\label{sec:moments}  

So far, we have studied the excursions of $(X_n, \eta_n)$ away from the line $\eta_n =0$ in terms
of the $\eta$-coordinate. The next stage is to study the behaviour, over an excursion,
of the $X$-coordinate. In particular, we estimate the moments of $Y_{n+1}-Y_n$,
with a view to later applying a Lamperti condition to determine the recurrence/transience of $(Y_n)$.  
First, we need estimates on 
 the maximum deviation  of $X_n$ during a single excursion:
\begin{equation}
\label{eqn:D-def}
 D_{n} := \max_{\tau_n \leq m \leq \tau_{n+1}} | X_m - X_{\tau_n} | ;
\end{equation} 
note that the distribution of $D_{n}$ given $X_{\tau_n} = x$ depends only on $x$ and not on $n$.

\begin{lemma}\label{lem:D-moments}
Suppose that condition \eqref{ass:X-diff-mom-p} holds for some $p>1$ and condition \eqref{ass:lim-q} holds. 
Then, for any $q \in (0,p)$, 
\[ \sup_x \Pr{[ D_{n} \geq d \mid X_{\tau_n} = x]} = O( d^{-q}), 
~~\text{and}~~  \sup_x \Exp{[D_{n}^q \mid X_{\tau_n}=x]} < \infty  . \]
\end{lemma}
\begin{proof}
Conditional on $X_{\tau_n} = x$, we have
\[ 
\begin{split}
\Pr{[D_{n} \geq d]} &\leq \Pr{[\sigma_{n+1}   \geq y]} + \Pr{[ D_{n} \geq d, \, \sigma_{n+1}   < y]}\\
&\leq C \re^{-c y} + \Pr{\bigg[\max_{\tau_n \leq m \leq \tau_n+y} | X_m - X_{\tau_n} | \geq d \bigg]},
\end{split}
\]
for all $d \geq 0$ and $y>0$, by Lemma \ref{lem:unif-irred}.
Here, 
\[
\begin{split}
\Pr{\left[\max_{\tau_n \leq m \leq \tau_n+y} | X_m - X_{\tau_n} | \geq d \right]} &\leq \Pr{\left[\max_{\tau_n \leq m \leq \tau_n+y} \sum_{\ell=\tau_n}^{m-1} | X_{\ell+1} - X_\ell | \geq d \right]}\\
&\leq \Pr{\left[ \bigcup_{\tau_n \leq \ell \leq \tau_n+y-1} \Bigl\{ |X_{\ell+1} - X_\ell | \geq \textstyle\frac{d}{y} \Bigr\} \right]}\\
&\leq y C_p \bigl(\textstyle\frac{d}{y}\bigr)^{-p},
\end{split}
\]
which follows from the inequalities of Boole and Markov and the fact that
\begin{align*}
& \Exp[ |X_{\ell+1} - X_\ell|^p  \mid X_{\tau_n} = x] \\
& {} \quad {} = \sum_{z,i} \Exp{[ |X_{\ell+1} - X_\ell|^p \mid  X_\ell = z, \eta_\ell = i]}\Pr{[X_\ell = z, \eta_\ell = i \mid X_{\tau_n} = x ]} \leq C_p,
\end{align*}
by assumption \eqref{ass:X-diff-mom-p}.
 Then, taking $y = d^{(p-q)/(1+p)}$,
 where $q \in (0,p)$, we obtain
  $\Pr{[D_{n} \geq d \mid X_{\tau_n}=x]} = O ( d^{-q} )$, as claimed.  The final claim follows from the fact that
\[
\Exp{[D_{n}^\alpha \mid X_{\tau_n}=x]} = \sum_{d=1}^\infty \Pr{[D_{n}^\alpha \geq d \mid X_{\tau_n}=x]} \leq \int_0^\infty \Pr{[D_{n}^\alpha \geq t \mid X_{\tau_n}=x]} \ud t,
\]
which is finite when $\alpha \in (0,q)$, where $q$ can be arbitrarily close to $p$.
\end{proof}

We are now in a position to calculate the moments of $Y_{n+1} - Y_n$.
The first case to consider is when, for each $i$, $\mu (x, i)$ is asymptotically $d_i$.

\begin{lemma}
\label{lem:moments-constant-drifts}
Suppose that condition \eqref{ass:X-diff-mom-p} holds for some $p> 1$, and conditions 
\eqref{ass:lim-q} and \eqref{ass:mu1-const} hold. 
 Then there exists $\eps >0$  such that
\begin{align}
\label{eqn:Y_one-plus-moments}
\sup_x \Exp{[ |Y_{n+1} - Y_n|^{1+\eps} \mid Y_n = x ]} & < \infty. 
\end{align}
 Also, as $x \to \infty$,
\begin{align}
\label{eqn:Y-drift-positive}
\Exp{[ Y_{n+1} - Y_n \mid Y_n = x ]} &=  \frac{1}{\pi(0)} \sum_{i \in S}  d_i \pi(i)  + o(1). \end{align}
\end{lemma}
\begin{proof}
First, note that $| Y_{n+1} - Y_n | = | X_{\tau_{n+1}} - X_{\tau_n} | \leq | D_n|$, a.s.,
where $D_n$ is given by \eqref{eqn:D-def}.
Then the statement \eqref{eqn:Y_one-plus-moments} follows from Lemma~\ref{lem:D-moments} with \eqref{ass:X-diff-mom-p} for $p > 1$.

It remains to prove \eqref{eqn:Y-drift-positive};
 by the time-homogeneity of $(X_n,\eta_n)$ 
and since $Y_n = X_{\tau_n}$,
it suffices to consider $\Exp_{x,0} [ X_\tau - X_0 ]$.
The Doob decomposition for $X_n$ is
\[
X_n - X_0 = M_n + \sum_{k=0}^{n-1} \Exp{[X_{k+1} - X_k \mid X_k,\eta_k ]},
\]
where $M_n$ is a martingale with $M_0 = 0$. Hence, by definition of $\mu_1(x,i)$,
 \begin{align*}
X_n - X_0  = M_n + \sum_{k=0}^{n-1} \mu_1(X_k,\eta_k)  = M_n + \sum_{i \in S} \sum_{k=0}^{n-1} \mu_1(X_k,i) \1{\{\eta_k = i\}}.
\end{align*}
Since $\Exp{\tau} < \infty$, and $\Exp{[| M_{n+1} - M_n|  \mid \F_n ]} \leq 2 \Exp{[|X_{n+1} - X_n| \mid \F_n ]} \leq 2C_1$, a.s.,
(by the $p=1$ case of~\eqref{ass:X-diff-mom-p}), the Optional Stopping Theorem gives $\Exp M_\tau  = M_0 = 0$.  Therefore,
\begin{equation}\label{eqn:X-moms-sum}
\Exp_{x,0} {[ X_\tau - X_0 ]} = 
\sum_{i \in S} \Exp_{x,0} { \left[ \sum_{k=0}^{\tau-1} \mu_1(X_k,i) \1{\{\eta_k = i\}} \right]}.
\end{equation}
Now, let $D = \max_{0 \leq k \leq \tau} | X_k - X_0 |$, and set $A_x = \{ D < x^\gamma \}$, for some $\gamma \in (0,1)$.
Note that, conditional on $X_0 = x$ and $\eta_0 = 0$, the random variable $D$ has the same distribution as the random variable $D_{n}$ defined
at \eqref{eqn:D-def} given $X_{\tau_n} = x$, so by Lemma~\ref{lem:D-moments} we have
\begin{equation}
\label{e:Axc}
\Pr_{x,0} [ A_x^\rc] = \Pr_{x,0} {[ D \geq x^\gamma ]} = O( x^{-\gamma  } ).
\end{equation} 

Now, given $X_0 = x$ and $A_x$, we have 
for all $0 \leq k \leq \tau$ that $X_k \geq x - x^\gamma \geq x/2$, say,
for all $x$ sufficiently large. Thus, by
\eqref{ass:mu1-const},
for any $\theta >0$, there exists $x_0 < \infty$ such that, given $X_0 = x \geq x_0$,
\[ \max_{i \in S} \max_{0 \leq k \leq \tau} \left| \mu_1 (X_k , i ) - d_i \right| \1 (A_x) \leq \theta   , \as \]
Since $\max_{x,i} |\mu_1 (x, i)| < \infty$ and $\max_i | d_i | < \infty$, it follows that there
exists a constant $C < \infty$ such that, given $X_0 = x \geq x_0$,
\[ \max_{i \in S} \max_{0 \leq k \leq \tau} \left| \mu_1 (X_k , i ) - d_i \right|   \leq \theta + C \1 ( A_x^\rc)  , \as \]
Hence, given $X_0 = x \geq x_0$,
\[ \left|  \sum_{k=0}^{\tau-1} \mu_1(X_k,i) \1{\{\eta_k = i\}}   -
  \sum_{k=0}^{\tau-1} d_i \1{\{\eta_k = i\}}  \right| \leq \theta \tau + C \tau \1 ( A_x^\rc) , \as \]
Here,
by 
the Cauchy--Schwarz inequality,
\[
\Exp_{x,0} { [ \tau \1{(A_x^\rc )} ] } \leq (\Exp_{x,0} { [ \tau^2 ] })^{1/2} (\Pr_{x,0} {[ A_x^\rc ]})^{1/2} = O ( x^{-\gamma/2}) , 
\]
using \eqref{e:Axc} and the fact that $\tau$ has all moments, by Lemma \ref{lem:unif-irred}.
So, for any $\delta >0$, we can choose $x_1 < \infty$ sufficiently large so that, given $X_0 = x \geq x_1$,
\[ \max_i \left| \Exp_{x,0} \left[ \sum_{k=0}^{\tau-1} \mu_1(X_k,i) \1{\{\eta_k = i\}}  \right] -
 \Exp_{x,0} \left[ \sum_{k=0}^{\tau-1} d_i \1{\{\eta_k = i\}}  \right] \right| \leq \delta . \]
 Together with Lemma~\ref{lem:occupation-limit} and \eqref{eqn:X-moms-sum} this yields \eqref{eqn:Y-drift-positive}.
 \end{proof}

\begin{lemma}
\label{lem:Y-moments-markov}
Suppose that condition \eqref{ass:X-diff-mom-p} holds for some $p>2$, and conditions \eqref{ass:lim-q} and \eqref{ass:mu-lamperti} hold. 
 Then  there exists $\eps > 0$   such that
\begin{equation}
\label{eqn:Y-moment-bound-lamperti}
\sup_x \Exp{[ |Y_{n+1} - Y_n|^{2+\eps} \mid Y_n = x ]}  < \infty.
\end{equation}
 Also, as $x \to \infty$,
\begin{align}
\Exp{[ Y_{n+1} - Y_n \mid Y_n = x ]} &=  \frac{1}{\pi(0)} \sum_{i \in S} \frac{c_i \pi(i)}{x} + o(x^{-1});\label{eqn:Y-1st-mom-lamperti}\\
\Exp{[ (Y_{n+1} - Y_n)^2 \mid Y_n = x ]} &= \frac{1}{\pi(0)} \sum_{i \in S} s_i^2 \pi(i) + o(1).\label{eqn:Y-2nd-mom-lamperti}
\end{align}
If, in addition \eqref{ass:lim-q+} and \eqref{ass:mu-lamperti+} hold, then there exists $\delta > 0$ such that
\begin{align}
\Exp{[ Y_{n+1} - Y_n \mid Y_n = x ]} &=  \frac{1}{\pi(0)} \sum_{i \in S} \frac{c_i \pi(i)}{x} + O(x^{-1-\delta});\label{eqn:Y-1st-mom-lamperti+}\\
\Exp{[ (Y_{n+1} - Y_n)^2 \mid Y_n = x ]} &= \frac{1}{\pi(0)} \sum_{i \in S} s_i^2 \pi(i) + O(x^{-\delta});\label{eqn:Y-2nd-mom-lamperti+}
\end{align}
\end{lemma}
\begin{proof}
First, since $|Y_{n+1} - Y_n| \leq D_{n}$, with $D_n$  as defined at \eqref{eqn:D-def}, 
and because Lemma~\ref{lem:D-moments} implies that $\sup_x \Exp[ (D_{n})^{2+\eps} \mid X_{\tau_n} = x]   < \infty$, 
\eqref{eqn:Y-moment-bound-lamperti} follows.

The proof of \eqref{eqn:Y-1st-mom-lamperti} and \eqref{eqn:Y-2nd-mom-lamperti} using \eqref{ass:lim-q} and \eqref{ass:mu-lamperti} and the proof of \eqref{eqn:Y-1st-mom-lamperti+} and \eqref{eqn:Y-2nd-mom-lamperti+} using \eqref{ass:lim-q+} and \eqref{ass:mu-lamperti+} are essentially the same, the only difference being in the error terms associated to each expression.  We present the proof of \eqref{eqn:Y-1st-mom-lamperti+} and \eqref{eqn:Y-2nd-mom-lamperti+}; it should be clear how to adapt the argument to prove \eqref{eqn:Y-1st-mom-lamperti} and \eqref{eqn:Y-2nd-mom-lamperti}.

We proceed as in the proof of Lemma~\ref{lem:moments-constant-drifts}.  Indeed, we follow the reasoning from the second paragraph of that proof through to equation \eqref{eqn:X-moms-sum}, giving
\[
\Exp_{x,0}{[ X_\tau - X_0 ]} = 
\sum_{i \in S} \Exp_{x,0} { \left[ \sum_{k=0}^{\tau-1} \mu_1(X_k,i) \1{\{\eta_k = i\}}  \right]},
\]
and we let $D = \max_{0 \leq k \leq \tau} | X_k - X_0 |$, and set $A_x = \{ D < x^\gamma \}$ as before, but now
 we require $\gamma \in (1/2, 1)$.
Note that, conditional on $X_0 = x$ and $\eta_0 = 0$, the random variable $D$ has the same distribution as the random variable $D_{n}$ defined
at \eqref{eqn:D-def} given $X_{\tau_n} = x$, so by Lemma~\ref{lem:D-moments} we have that $\Pr_{x,0} {[ D \geq d  ]} = O( d^{-p'} )$ for some $p' > 2$ since $\tau$ has all moments and \eqref{ass:X-diff-mom-p} holds for some $p>2$.

Now, given $X_0 = x$ and $A_x$, we have $|X_k - x| \leq D < x^\gamma$ for $k \leq \tau$, so that, by \eqref{ass:mu-lamperti+},
\[
\left| \mu_1(X_k,i) - \frac{c_i}{x} \right| \leq \left| \frac{c_i}{X_k} - \frac{c_i}{x} \right| + o( (x-x^\gamma)^{-1-\delta_1}) = O(x^{\gamma-2}) + O(x^{-1-\delta_1}),
\]
uniformly for $0 \leq k \leq \tau$.  Therefore $\mu_1(X_k,i) \1{(A_x)} = (c_i/x + O(x^{\gamma-2}) + O(x^{-1-\delta_1}) )\1{(A_x)}$, which means that $\mu_1(X_k,i) = c_i/x + O(x^{\gamma-2}) + O(x^{-1-\delta_1}) + O(1)\1{(A_x^\rc)}$.
So,
\[
\Exp_{x,0}  \left[ \sum_{k=0}^{\tau-1} \mu_1(X_k,i) \1{\{\eta_k = i\}}  \right]
=\Exp_{x,0} { \left[ \left(\frac{c_i}{x} + O(x^{\gamma-2}) + O(x^{-1-\delta_1}) + O(1)\1{(A_x^\rc)} \right) \sum_{k=0}^{\tau-1} \1{\{\eta_k = i\}}  \right]},
\]
where the implicit constants are uniform in $x$ and in $i$.
By \eqref{ass:lim-q+} and the second statement in Lemma~\ref{lem:occupation-limit}, we have that
\begin{equation}
\label{eq:occupation_with_error}
 \Exp_{x, 0} {\left[\sum_{k=0}^{\tau-1} \1{\{ \eta_k = i\} }\right]} = \frac{\pi(i)}{\pi(0)}  + O (x^{-\delta'} ),
\end{equation}
for some $\delta' > 0$, so
\[
\Exp_{x,0} {[ X_\tau - X_0 ]} = \frac{1}{\pi(0)} \sum_{i \in S} \frac{c_i \pi(i)}{x} + O(x^{-1-\delta'}) + O(x^{\gamma-2}) + O(x^{-1-\delta_1}) + O(1) \Exp_{x,0} { [ \tau \1{(A_x^{\rm c})} ] }.
\]
Here, by H\"older's inequality, for all $r,s > 0$ with $r^{-1}+s^{-1} =1$,
\[
\Exp_{x,0} { [ \tau \1{(A_x^\rc )} ] } \leq (\Exp_{x,0} { [ \tau^r ] })^{1/r} (\Pr_{x,0} {[ A_x^\rc  ]})^{1/s}.
\]
Since $\tau$ has all moments, we can take $s = p'/2 > 1$, so that $\Exp_{x,0} { [ \tau \1{(A_x^\rc )} ] } = O(x^{-2\gamma})$.
Then, since $\gamma \in (1/2,1)$, $\delta_1 > 0$ and $\delta' > 0$ we have, for some $\delta'' > 0$,
\[
\Exp{[Y_{n+1} - Y_n \mid Y_n=x]} = \frac{1}{\pi(0)} \sum_{i \in S} \frac{c_i \pi(i)}{x} + O(x^{-1-\delta''}).
\]

To calculate the second moment of $X_\tau - X_0$, we will make repeated use of the algebraic identity $a^2 - b^2 = (a-b)^2 + 2b(a-b)$, which will help to simplify the calculations that follow.
Taking the Doob decomposition for $X_n^2$, we write
\[
\begin{split}
X_n^2 - X_0^2 &= M_n + \sum_{k=0}^{n-1} \Exp{[ X_{k+1}^2 - X_k^2 \mid X_k,\eta_k ]}\\
&= M_n + \sum_{k=0}^{n-1} \left( \Exp{[ (X_{k+1} - X_k)^2 \mid X_k,\eta_k ]} + 2X_k \Exp{[X_{k+1}-X_k \mid X_k,\eta_k ]} \right)\\
&= M_n + \sum_{k=0}^{n-1} \left( \mu_2(X_k, \eta_k) + 2X_k \mu_1(X_k,\eta_k) \right)\\
&= M_n + \sum_{i \in S} \sum_{k=0}^{n-1} ( s_i^2 + 2c_i + O(X_k^{-\delta_1}) )\1{\{\eta_k=i\}},
\end{split}
\]
by \eqref{ass:mu-lamperti+}, 
where $M_n$ is a martingale satisfying $M_0 = 0$.
Moreover,  given $X_0 = x$,
\[
\begin{split}
| M_{n \wedge \tau} | &\leq | X_{n \wedge \tau}^2 - X_0^2|  + C \tau\\
&= (X_{n \wedge \tau} - X_0)^2 + 2X_0 | X_{n \wedge \tau} - X_0| + C \tau\\
&\leq D^2 + 2xD + C\tau,
\end{split}
\]
where $D = \max_{0\leq k \leq \tau}|X_k-X_0|$ is as defined earlier, and $C < \infty$ is a constant.  Thus, $M_{n \wedge \tau}$ is uniformly integrable (in $n$) and so by the Optional Stopping Theorem $\Exp M_\tau = M_0 = 0$. Therefore,
\[ 
\Exp_{x,0} {[ X_\tau^2 - X_0^2  ]} = \sum_{i \in S} (s_i^2 + 2c_i)\Exp_{x,0} {\left[\sum_{k=0}^{\tau-1} \1{\{\eta_k=i\}}  \right]} 
  + \Exp_{x,0}{\left[ \sum_{k=0}^{\tau-1} O(X_k^{-\delta_1}) \right]}.
\]
As in the calculation of the first moment, we can bound the error term by bootstrapping on the event $A_x$: writing $O(X_k^{-\delta_1}) = O(x^{-\delta_1}) + O(1) \1{(A_x^{\rm c})}$, we get
\begin{align*}
\Exp_{x,0} {\left[ \sum_{k=0}^{\tau-1} O(X_k^{-\delta_1})  \right]} & = O(x^{-\delta_1}) + O(1)\Exp_{x,0} {[\tau \1{(A_x^\rc)} ] } \\
& = O(x^{-\delta_1}) + O(x^{-2\gamma}),
\end{align*}
as above, 
and therefore, by \eqref{eq:occupation_with_error},
\[
\Exp_{x,0} {[ X_\tau^2 - X_0^2 ]} = \frac{1}{\pi(0)} \sum_{i \in S} (s_i^2 + 2c_i)\pi(i) + O(x^{-\delta'}) + O(x^{-\delta_1}) + O(x^{-2\gamma}).
\]
Now we use $X_\tau^2 - X_0^2 = (X_\tau - X_0)^2 + 2X_0(X_\tau - X_0)$ to get
\[
\Exp_{x,0} {[ (X_\tau - X_0)^2  ]} = \frac{1}{\pi(0)} \sum_{i \in S} s_i^2 \pi(i) + O(x^{-\delta'''}),
\]
for some $\delta''' > 0$.  Finally, taking $\delta = \min\{\delta'',\delta'''\}$ yields \eqref{eqn:Y-1st-mom-lamperti+} and \eqref{eqn:Y-2nd-mom-lamperti+}, as required.
\end{proof}

\section{Proofs of main results}
\label{sec:proofs}

\subsection{Recurrence classification}

To prove Theorems \ref{thm:constant-drifts} and \ref{thm:lamperti_rough_then_sharp}, we use the increment moment estimates from Section~\ref{sec:moments} together with
some Foster--Lamperti conditions to classify the process $(Y_n)$, and then deduce the
classification for $(X_n)$ from the equivalence results in Section \ref{sec:recurrence-transience}.

For Theorem \ref{thm:lamperti_rough_then_sharp}, under Lamperti-type drift assumptions,  we apply 
 the following classification result.

\begin{lemma}[Lamperti]\label{lem:Lam}
Let $(Z_n )$ be an irreducible time-homogeneous Markov chain on $\ZP$.
Suppose that there exists $\eps>0$ such that
\begin{align}
\label{eqn:two-plus-moments}
\sup_z \Exp{[ |Z_{n+1}-Z_n|^{2+\eps} \mid Z_n = z ]} & < \infty ; \\
\label{eqn:positive-variance}
\liminf_{z \to \infty} \Exp{[ |Z_{n+1}-Z_n|^{2} \mid Z_n = z ]} & > 0 .\end{align}
Let $\mu_k(z) = \Exp{[ (Z_{n+1} - Z_n)^k \mid Z_n = z ]}$.  
\begin{itemize}
\item If $\liminf_{z \to \infty} ( 2z \mu_1(z) - \mu_2(z) )> 0$,  then $Z_n$ is transient.
\item If $|2z \mu_1(z)| \leq  \mu_2(z) + O(z^{-\delta})$, for some $\delta > 0$,  then $Z_n$ is null-recurrent.
\item If $\limsup_{z \to \infty} (2z \mu_1 (z) + \mu_2(z) ) < 0$, then $Z_n$ is positive-recurrent.
\end{itemize}
\end{lemma}

Lemma \ref{lem:Lam} is essentially due to Lamperti \cite{lamp1,lamp3}, although the form given here  is taken from Menshikov \emph{et al.}~\cite[Theorem 3]{mai}.
The conditions for recurrence and transience are contained in Theorem 3.2 of \cite{lamp1}, and the condition
for positive-recurrence is contained in Theorem 2.1 of \cite{lamp3}. The condition for null-recurrence here is slightly sharper than
Lamperti's original results~\cite{lamp3}.

\begin{proof}[Proof of Theorem \ref{thm:lamperti_rough_then_sharp}.]
We apply Lemma~\ref{lem:Lam} to classify $Z_n = Y_n$, and thus, by Lemma~\ref{lem:X-Y}, classify $X_n$. 
First, assuming \eqref{ass:X-diff-mom-p} for some $p>2$, \eqref{ass:lim-q} and \eqref{ass:mu-lamperti}, by Lemma~\ref{lem:Y-moments-markov} it is clear that
\eqref{eqn:two-plus-moments} and \eqref{eqn:positive-variance} hold for $Z_n = Y_n$.
Furthermore,
\[
\liminf_{ x \to \infty} 2x \Exp{[  Y_{n+1} - Y_n  \mid Y_n = x ]} = 
\limsup_{x \to \infty} 2x \Exp{[  Y_{n+1} - Y_n  \mid Y_n = x ]} = \frac{1}{\pi(0)}\sum_{i \in S} 2c_i \pi(i),
\]
and
\[
\liminf_{ x \to \infty} \Exp{[ |Y_{n+1} - Y_n|^2 \mid Y_n = x ]} = \limsup_{x \to \infty} \Exp{[ |Y_{n+1} - Y_n|^2 \mid Y_n = x ]} = \frac{1}{\pi(0)}\sum_{i \in S} s_i^2 \pi(i).
\]
By Lemma~\ref{lem:Lam}, $\sum_{i \in S} (2c_i - s_i^2)\pi(i) > 0$ implies transience, 
while $\sum_{i \in S} (2c_i + s_i^2)\pi(i) < 0$ implies positive-recurrence.  When  $| \sum_{i \in S} 2c_i \pi (i) | < \sum_{i \in S} s_i^2 \pi(i)$, we have
\[
\lim_{x \to \infty}( | 2x  \Exp{[  Y_{n+1} - Y_n  \mid Y_n = x ]} | - \Exp{[ |Y_{n+1} - Y_n|^2 \mid Y_n = x ]}) < 0,
\]
which means the middle condition of Lemma~\ref{lem:Lam} holds for \emph{any} $\delta > 0$, and therefore $Y_n$ is null-recurrent.

Now suppose that \eqref{ass:lim-q+} and \eqref{ass:mu-lamperti+} also hold. Then, by Lemma~\ref{lem:Y-moments-markov}, we have
\[
\begin{split}
 | 2x  \Exp{[  Y_{n+1} - Y_n  \mid Y_n = x ]} | &- \Exp{[ |Y_{n+1} - Y_n|^2 \mid Y_n = x ]}
\\
 &\quad\quad = \frac{1}{\pi(0)} \left( \left| \sum_{i \in S} 2c_i \pi (i) \right| - \sum_{i \in S} s_i^2 \pi(i)\right) + O(x^{-\delta})
\end{split}
\]
for some $\delta > 0$, which means that  $| \sum_{i \in S} 2c_i \pi (i) | = \sum_{i \in S} s_i^2 \pi(i)$ implies that $(Y_n)$ is null-recurrent, 
completing the classification of $(Y_n)$ and therefore of $(X_n)$.
\end{proof}

For Theorem \ref{thm:constant-drifts}
 we will apply the following condition.

\begin{lemma}
\label{lem:transience}
Let $(Z_n )$ be an irreducible time-homogeneous Markov chain on $\ZP$.
For $(Z_n)$ to be transient, it is sufficient that  there exists $\eps>0$ such that
\begin{align}
\label{eqn:one-plus-moments}
\sup_z \Exp [ | Z_{n+1} -Z_n |^{1+\eps} \mid Z_n = z ] < \infty , ~\text{and} \\
\label{eqn:positive-drift}
\liminf_{z \to \infty} \Exp [ Z_{n+1} - Z_n \mid Z_n = z ] > 0 .
\end{align}
\end{lemma}

We omit the proof of Lemma \ref{lem:transience},
which is similar to the proof of Lemma \ref{lem:Lam} and relies 
on demonstrating the existence of a suitable
Lyapunov function with negative drift outside a bounded set,
using Taylor's formula and some careful truncation.

\begin{proof}[Proof of Theorem \ref{thm:constant-drifts}.]
Consider the Markov chain $(Y_n)$. Under the conditions of part (i) of the theorem,
Lemma~\ref{lem:moments-constant-drifts} implies that the hypotheses of Lemma~\ref{lem:transience} hold
for $Z_n = Y_n$, so that $(Y_n)$ is transient. Hence, by Lemma~\ref{lem:X-Y}, $(X_n)$ is also transient. 

As mentioned after the statement, part (ii) was obtained by Falin \cite{falin1}.
Our results furnish a different proof: Lemma~\ref{lem:moments-constant-drifts} gives positive-recurrence for $(Y_n)$
by  Foster's criterion (e.g.\ Theorem 2.2.3 of \cite{fmm}), so, by Lemma~\ref{lem:X-Y}, $(X_n)$ is also positive-recurrent. 
\end{proof}

\subsection{Convergence in distribution}
\label{sec:weak_convergence}

The first step in the proof of Theorem \ref{thm:weak_limit}
is to apply a result of Lamperti \cite{lamp2} to obtain
a weak limit for the embedded Markov chain $(Y_n)$.  Recall the distribution function $F_{\alpha,\theta}$ as defined at \eqref{eqn:Fdef}.

\begin{lemma}
\label{lem:embedded_weak_limit}
Suppose $(X_n,\eta_n)$ is a Markov chain satisfying \eqref{ass:X-diff-mom-p} for some $p>4$, \eqref{ass:lim-q} and \eqref{ass:mu-lamperti}.
Suppose that the matrix $q$ appearing in \eqref{ass:lim-q} is aperiodic.
Suppose also that 
  $\sum_{i \in S} (2c_i + s_i^2)\pi(i) > 0$. Define $\alpha$ and $\theta$ as at \eqref{eq:alpha_theta}.
Then, for any $x \in \RP$,
\[  \lim_{n \to \infty} \Pr \left[ n^{-1/2} Y_n \leq x \right] = F_{\alpha,\theta} \big( x \sqrt{\pi (0 )} \big) .\]
\end{lemma}
\begin{proof}
If \eqref{ass:X-diff-mom-p} holds for some $p>4$, then a consequence of Lemma \ref{lem:D-moments} is that
\[ \sup_x \Exp \left[ | Y_{n+1} - Y_n |^4 \mid Y_n  = x \right] < \infty .\]
Now we apply Theorem 2.1 of \cite{lamp2} to the Markov chain $(Y_n)$,
using the increment moment estimates of Lemma~\ref{lem:Y-moments-markov}
and noting the remark preceding the theorem in \cite{lamp2}, to obtain
\[ \lim_{n \to \infty} \Pr \left[ n^{-1/2} Y_n \leq x \right] = F_{\alpha,\theta / \pi (0)} (x) .\]
Taking $x = \beta x$ in \eqref{eqn:Fdef} and using the change of variable $v =u/\beta$ one observes
the scaling relation, valid for any $\beta > 0$,
$F_{\alpha, \theta} ( \beta x ) = F_{\alpha, \theta/\beta^2} (x)$,
which implies the result.
\end{proof}

\begin{remark}
If in addition \eqref{ass:lim-q+} and \eqref{ass:mu-lamperti+} hold, then in the case $\sum_{i \in S} (2c_i + s_i^2)\pi(i) = 0$
it follows from Lemma 2.1 of \cite{lamp2} that $n^{-1/2} Y_n \to 0$ in probability; cf Remark \ref{rmk:weak_limit}(iii).
\end{remark}

The next goal is to deduce from the weak limit for $Y_n$ a weak limit for $X_n$.
To do so, we need (i) to control the value of the process $(X_n)$ between successive observations of the embedded process, and (ii) to account for the change of time.
First we address point~(i).
For each $n \in \ZP$, let $N(n) := \max \{ k : \tau_k \leq n \}$,
so that $\tau_{N(n)} \leq n < \tau_{N(n) +1}$.

\begin{lemma}
\label{lem:weak_discrepancy}
Suppose that $(X_n,\eta_n)$ satisfies \eqref{ass:X-diff-mom-p} for some $p>2$ and \eqref{ass:lim-q}.
Then, as $n \to \infty$, $n^{-1/2} | X_n - X_{\tau_{N(n)} } | \to 0$ in probability.
\end{lemma}
\begin{proof}
Since $\sigma_{k} \geq 1$, we have that $N(n) \leq n$, a.s. Hence $| X_n - X_{\tau_{N(n)}} | \leq
\max_{k \leq n} D_k$, where $D_k$ is as defined at \eqref{eqn:D-def}. Thus
it suffices to show that $n^{-1/2} \max_{k \leq n} D_k \to 0$ in probability. For any
$\gamma >0$, we have
\begin{align*}
\max_{k \leq n} D_k & \leq n^\gamma + \max_{k \leq n} \left( D_k \1 \{ D_k > n^\gamma \} \right) \\
& \leq n^\gamma + \sum_{k=0}^n D_k \1 \{ D_k > n^\gamma \} .
\end{align*}
Since  \eqref{ass:X-diff-mom-p} holds for   $p>2$,
  Lemma \ref{lem:D-moments} shows
  $\max_k \Exp [ D_k^q ] < \infty$ for some $q >2$, so
\[ \Exp \left[ D_k \1 \{ D_k > n^\gamma \} \right] \leq \Exp \left[ D_k^q  D_k^{1-q} \1 \{ D_k > n^\gamma \} \right] \leq n^{-\gamma(q-1)} \Exp \left[ D_k^q \right] .\]
It follows that
\[ n^{-1/2} \Exp \max_{k \leq n} D_k  = O (n^{\gamma-(1/2)}) + O (n^{(1/2)-\gamma(q-1)} ) ,\]
which is $o(1)$ provided we choose (as we may) $\frac{1}{2(q-1)} < \gamma < \frac{1}{2}$.
Thus $n^{-1/2} \max_{k \leq n} D_k \to 0$ in $L^1$, and hence in probability.
\end{proof}

Next we turn to point (ii) mentioned above. For our purposes, the following renewal-type result will suffice.

\begin{lemma}
\label{lem:N_limit}
Suppose $(X_n,\eta_n)$ is a Markov chain satisfying \eqref{ass:X-diff-mom-p} for some $p>2$, \eqref{ass:lim-q} and \eqref{ass:mu-lamperti}.
Suppose also that 
  $\sum_{i \in S} (2c_i + s_i^2)\pi(i) > 0$.

Then, as $n \to \infty$, $n^{-1} N(n) \to \pi (0)$ in probability.
\end{lemma}
\begin{proof}
Under the conditions of the lemma,  Theorem \ref{thm:lamperti_rough_then_sharp} shows that $X_n$ (and hence
$Y_n$)
is \emph{null}, i.e., null-recurrent or transient. In particular, 
for any $x \geq 0$, 
 \begin{equation}
 \label{eqn:Y-null}
  \lim_{n \to \infty} \Exp \Bigg[ \frac{1}{n} \sum_{k=0}^{n-1}    \1 \{ X_{\tau_k} \leq x \} \Bigg] = 0. \end{equation}
 
We use an   extension of the coupling given in Lemma \ref{lem:coupling} to multiple excursions. 
We construct on the same probability space $(X_n, \eta_n)$ together with a sequence $(\eta^\star_{k,n} )$
of  copies (for $k \in \ZP$)
of the Markov chain $(\eta^\star_n)$ as follows. At each $\tau_k$, $k \in \ZP$,
start $(\eta^\star_{k,n})_{n \geq 0}$, an independent
 copy of $(\eta^\star_n)_{n \geq 0}$, from $\eta^\star_{k,0} = \eta_{\tau_k} = 0 \in S$, coupled to $(\eta_n)_{n \geq \tau_k}$ as described
in Lemma \ref{lem:coupling}; denote by $\sigma^\star_{k+1}$ the number of steps
until   $ \eta^\star_{k,n} $ returns to $0$. 

Extending the notation $E_n$ defined at \eqref{eq:E_n_def},
 we write $E_{k,n} = \cap_{0 \leq \ell \leq n} \{ \eta_{\tau_k +\ell} = \eta^\star_{k, \ell} \}$,
 the event that the coupling started at $\tau_k$ succeeds for $n$ steps.

Now we use this coupling construction and the null property \eqref{eqn:Y-null} to show that $n^{-1} \tau_n \to \pi (0)^{-1}$ in probability.
For $s >0$, denote $\chi_s (x) := x \1 \{ x \leq s\}$. Note that
\begin{align*}
\Bigg| \frac{1}{n} \sum_{k=0}^{n-1} \chi_s \left( \sigma_{k+1} \right)  -  \frac{1}{n} \sum_{k=0}^{n-1}   \sigma_{k+1}  
\Bigg|
 \leq \frac{1}{n} \sum_{k=0}^{n-1}   \sigma_{k+1}   \1 \{ \sigma_{k+1} > s \} .\end{align*}
 Here $\Exp [  \sigma_{k+1}   \1 \{ \sigma_{k+1} > s \} ] \leq s^{-1} \Exp [ \sigma_{k+1}^2 ]$, say, so that,  by Lemma \ref{lem:unif-irred},
 \[ \lim_{s \to \infty} \sup_k \Exp \left[   \sigma_{k+1}   \1 \{ \sigma_{k+1} > s \} \right] = 0.\]
  A similar argument holds for $\sigma_{k+1}^\star$. Hence, for any $\eps >0$, there exists $s_0 < \infty$ such that
 \begin{align}
 \label{eqn:excursion-truncation}
 \Exp \Bigg|
\Bigg( \frac{1}{n} \sum_{k=0}^{n-1}   \sigma_{k+1}    -  \frac{1}{n} \sum_{k=0}^{n-1}    \sigma_{k+1}^\star   
\Bigg)   - 
\Bigg(  \frac{1}{n} \sum_{k=0}^{n-1} \chi_s \left( \sigma_{k+1} \right)    -  \frac{1}{n} \sum_{k=0}^{n-1} \chi_s \left( \sigma_{k+1}^\star \right)   \Bigg)
\Bigg|
 \leq \eps , \end{align}
 for all $s \geq s_0$ and all $n$. 
 On the event $E_{k,s}$ (the coupling started at $\tau_k$ succeeds for $s$ steps)
 we have $\chi_s (\sigma_{k+1} ) = \chi_s (\sigma^\star_{k+1} )$. 
 Then, for any $x > 0$,
 \[ 
 \left| \chi_s \left( \sigma_{k+1} \right)    -    \chi_s \left( \sigma_{k+1}^\star \right)  \right|
 \leq s \1 ( E_{k,s}^{\rm c} ) \1 \{ X_{\tau_k} > x \} + s \1 \{ X_{\tau_k} \leq x \} .\]
Now
\begin{align*}
 \Pr [ E_{k,s}^{\rm c} \cap \{   X_{\tau_k} > x \} ] & \leq \sup_{y >x} \Pr [  E_{k,s}^{\rm c}  \mid X_{\tau_k} =y , \, \eta_{\tau_k} = \eta^\star_{k,0} = 0 ] \\
 & = \sup_{y > x} \Pr [  E_{s}^{\rm c}  \mid X_0 =y , \, \eta_0 = \eta^\star_0 = 0 ] .\end{align*}
 So for   fixed $s \geq s_0$, Lemma \ref{lem:coupling} shows we may choose $x \geq x_0$ large enough such that, 
 \[ \Exp \frac{1}{n} \sum_{k=0}^{n-1} s \1 ( E_{k,s}^{\rm c} ) \1 \{ X_{\tau_k} > x \} \leq \eps , \]
 for all $n$. 
Combining this with the null property \eqref{eqn:Y-null}, we obtain that, for fixed $s \geq s_0$,
\[ \limsup_{n \to \infty} \Exp \Bigg[  \frac{1}{n} \Bigg| \sum_{k=0}^{n-1} \chi_s (\sigma_{k+1} ) - \sum_{k=0}^{n-1} \chi_s (\sigma^\star_{k+1} ) \Bigg|  \Bigg] \leq \eps .\]
Thus with \eqref{eqn:excursion-truncation} we conclude that
 \[ \limsup_{n \to \infty} 
 \Exp \Bigg| \frac{1}{n} \sum_{k=0}^{n-1}   \sigma_{k+1}    -  \frac{1}{n} \sum_{k=0}^{n-1}   \sigma_{k+1}^\star   \Bigg|
 \leq 2 \eps .\]
 Since $\eps >0$ was arbitrary,  and $\sigma_{k+1}^\star$ are i.i.d.\ random variables with mean $\pi(0)^{-1}$, it follows that
 $n^{-1} \tau_n \to \pi (0)^{-1}$ in probability.
 
 The claimed result now follows by inverting the law of large numbers: for example,
 \begin{align*}
 \Pr \left[ n^{-1} N(n) > \pi(0) + \eps \right] \leq \Pr \left[ \tau_{\lceil (\pi(0)+\eps) n \rceil} \leq n \right]
 \leq \Pr \left[ \frac{\tau_{\lceil (\pi(0)+\eps) n \rceil}}{\lceil (\pi(0)+\eps) n \rceil} \leq \frac{1}{\pi(0) + \eps} \right] ,\end{align*}
 which tends to $0$ as $n \to \infty$ for any $\eps>0$; similarly in the other direction.
\end{proof}

In the proof of Theorem \ref{thm:weak_limit} we will use two facts about convergence in distribution that we now
recall (see e.g.\ \cite[p.\ 73]{durrett}).
First, if sequences of random variables
$\xi_n$ and $\zeta_n$ are such that $\zeta_n \to \zeta$ in distribution for some random variable $\zeta$ and $|\xi_n -\zeta_n| \to 0$ in probability,
then $\xi_n \to \zeta$ in distribution (this is \emph{Slutsky's theorem}). Second,
if $\zeta_n \to \zeta$ in distribution and $\alpha_n \to \alpha$ in probability, then $\alpha_n \zeta_n \to \alpha \zeta$ in distribution.

\begin{proof}[Proof of Theorem \ref{thm:weak_limit}.]
First, since $n^{-1} N(n) \to \pi (0)$ in probability (Lemma \ref{lem:N_limit}), 
\[ \lim_{n \to \infty} \Pr \Bigg[ \frac{X_{\tau_{N(n)}} }{\sqrt{N(n)}} \cdot \sqrt{\frac{N(n)}{n}} \leq x \Bigg]
= \lim_{n \to \infty} \Pr \Bigg[ \frac{X_{\tau_{N(n)}} }{\sqrt{N(n)}} \leq \frac{x}{\sqrt{\pi(0)}} \Bigg]
= F_{\alpha, \theta} (x) ,\]  
by Lemma \ref{lem:embedded_weak_limit} and the fact that $\lim_{n \to \infty} N(n) = \infty$ a.s.
Together with Lemma \ref{lem:weak_discrepancy} and Slutsky's theorem, this shows that
\begin{equation}
\label{eqn:X-weak-convergence}
 \lim_{n \to \infty} \Pr [ n^{-1/2} X_n \leq x ] =  \lim_{n \to \infty} \Pr [ n^{-1/2} X_{\tau_{N(n)}} \leq x ] =  F_{\alpha, \theta} (x) .\end{equation}
Next we prove the joint convergence of $(X_n, \eta_n)$.
For $m \in \ZP$, let 
$R_{n,m} = n^{-1/2} | X_{n-m} - X_n |$.
Then, by the $p=1$ case of \eqref{ass:X-diff-mom-p}, we have $\Exp [R_{n,m}
] \leq C m n^{-1/2}$ for some finite constant $C$. Hence, for fixed $m$, as $n \to \infty$,
$R_{n,m} \to 0$ in $L^1$ and hence in probability.

Fix $x \in (0,\infty)$. Then, for any $\eps \in (0,x)$,
\begin{align*}
\Pr [ n^{-1/2} X_n > x, \, \eta_n = k ] \leq \Pr [ n^{-1/2} X_{n-m} > x-\eps , \, \eta_n = k ] + \Pr [ R_{n,m} \geq \eps ] .\end{align*}
Here
\begin{align}
\label{eqn:joint_upper}
\Pr [ n^{-1/2} X_{n-m} > x-\eps , \, \eta_n = k ] & = \sum_{y : n^{-1/2} y > x -\eps} \Pr [ X_{n-m} = y ] \Pr [ \eta_n = k \mid X_{n-m} = y ] .\end{align}
Again we use the coupling of Lemma \ref{lem:coupling} and the notation $E_n$ from \eqref{eq:E_n_def}. 
Note that
\[ \left| \Pr_{y,i,i} [ \eta_m = k ] - \pi (k) \right| \leq \Pr_{y,i,i} [ E_m^\rc ] + \left| \Pr_{y,i,i} [ \eta^\star_m = k ] - \pi (k) \right| .\]
Here, since $(\eta^\star_n)$ is an aperiodic, irreducible finite Markov chain
with stationary distribution $\pi$,
  $\Pr_{y,i,i} [ \eta^\star_m = k ] = \Pr[ \eta^\star_m = k \mid \eta^\star_0 = i]$
converges (uniformly over $i$ and $y$) to $\pi (k)$ as $m \to \infty$. So, for any  $\delta >0$,
we may choose $m_0 < \infty$ such that, for all $i$ and all $y$,
\[ \left| \Pr_{y,i} [ \eta_{m_0} = k ] - \pi (k) \right| \leq \Pr_{y,i,i} [ E_{m_0}^\rc ] + \delta .\]
By Lemma \ref{lem:coupling}, we may then choose $y_0 < \infty$ large enough so that, for all $y \geq y_0$,
\[ \left| \Pr  [ \eta_{m_0} = k \mid X_0 = y ] - \pi (k) \right| \leq 2 \delta .\]
Now taking $n$ large enough so that $(x-\eps) n^{1/2} > y_0$,
it follows from \eqref{eqn:joint_upper} that  
\begin{align*}
\Pr [ n^{-1/2} X_n > x, \, \eta_n = k ] \leq \Pr [ R_{n,m_0} \geq \eps ] + ( \pi (k) + 2\delta )   \Pr [ n^{-1/2} X_{n-m_0} > x -\eps ]. \end{align*}
We now let $n \to \infty$ and apply \eqref{eqn:X-weak-convergence} to obtain
\[ \limsup_{n \to \infty} \Pr [ n^{-1/2} X_n > x, \, \eta_n = k ] \leq ( \pi (k) + 2\delta ) \left( 1 - F_{\alpha, \theta} (x-\eps) \right) .\]
Since $\eps>0$ and $\delta>0$ were arbitrary, and $F_{\alpha, \theta}$ is continuous, it follows that
\[ \limsup_{n \to \infty} \Pr [ n^{-1/2} X_n > x, \, \eta_n = k ] \leq  \pi (k) \left( 1 -   F_{\alpha, \theta} (x ) \right) , ~~\text{for all}~ x \in (0,\infty).\]
A similar argument in the other direction, starting from the inequality 
\begin{align*}
\Pr [ n^{-1/2} X_n > x, \, \eta_n = k ] \geq \Pr [ n^{-1/2} X_{n-m} > x+\eps , \, \eta_n = k ] - \Pr [ R_{n,m} \geq \eps ]  \end{align*}
yields the complementary $\liminf$ statement, so that 
\begin{equation}
\label{eq:upper_tail_limit}
 \lim_{n \to \infty} \Pr [ n^{-1/2} X_n > x, \, \eta_n = k ] =  \pi (k) \left( 1 -  F_{\alpha, \theta} (x ) \right) , ~~\text{for all}~ x \in (0,\infty). \end{equation}
The statement in the theorem now follows from the fact that, by \eqref{eq:upper_tail_limit},
\[ \lim_{n \to \infty} \Pr [ n^{-1/2} X_n \leq x, \, \eta_n = k ] = \lim_{n \to \infty} \Pr[ \eta_n = k] - \pi (k) \left( 1 -  F_{\alpha, \theta} (x ) \right), \]
where $\lim_{n \to \infty} \Pr[ \eta_n = k] = \pi(k)$ by taking $x \downarrow 0$ in \eqref{eq:upper_tail_limit}.
\end{proof}

\appendix

\section{Proof of coupling lemma}
\label{sec:technical_appendix}

In this appendix we give the deferred technical proof of our coupling result, Lemma~\ref{lem:coupling}.

\begin{proof}[Proof of Lemma \ref{lem:coupling}.]
As commented on earlier, the proof follows an almost standard coupling argument.  Indeed, since the first two statements of the lemma will be satisfied for any coupling of $(X_n,\eta_n)$ and $(\eta^\star_n)$ on a common probability space, in order to also prove (\ref{eqn:eta-eta-star-coupling}/\ref{eqn:eta-eta-star-coupling2}) it makes sense to use a maximal coupling of $\eta_n$ and $\eta^\star_n$, which we will construct in a step-wise fashion.  For us, the condition that $q_x(i,j)$ has a limit as $x \to \infty$ means that the probability of decoupling at any step will be small, provided that $X_n$ stays sufficiently large. This introduces some complications to the standard coupling arguments, as we will need to keep control of the variation of $X_n$.

We construct the Markov chain $(X_n,\eta_n,\eta^\star_n)$ by describing a single step:
\begin{itemize}
\item If $\eta_n \neq \eta^\star_n$ then produce $(X_{n+1},\eta_{n+1})$ from $(X_n, \eta_n)$
 according to the transition probabilities $p(x,i,y,j)$, and produce $\eta^\star_{n+1}$ from $\eta^\star_n$
independently according to the transition probabilities $q(i,j)$.
\item Otherwise, given $\eta_n = \eta^\star_n = i$ and $X_n = x$, we use a maximal coupling (see, for example, Lindvall~\cite[pp.\ 18--20]{lindvall}) to produce $(\eta_{n+1}, \eta^\star_{n+1})$ via
\[
\Pr{[\eta_{n+1} = j, \eta^\star_{n+1} = k]} = \begin{cases}
\min\{ q_x(i,j) , q(i,k) \} &\text{for $j=k$},\\
\\
\displaystyle\frac{ (q_x(i,j) - q(i,j))^+(q(i,k)-q_x(i,k))^+ }{ \frac{1}{2} \sum_{\ell \in S} |q_x(i,\ell)-q(i,\ell)| }&\text{for $j \neq k$}.
\end{cases}
\]
Then, given $\eta_{n+1} = j$ we produce $X_{n+1}$ via
\[
\Pr{[X_{n+1} = y \mid \eta_{n+1}  =j ]} = \frac{p(x,i,y,j)}{\sum_{z \in \ZP} p(x,i,z,j)}.
\]
\end{itemize}
It is a simple matter to check that we have constructed a valid coupling of $(X_n,\eta_n)$ and $\eta^\star_n$.  Indeed, making use of the fact that
\[
\sum_{\ell \in S} | q_x(i,\ell) - q(i,\ell)| = \sum_{\ell \in S} (q_x(i,\ell) - q(i,\ell))^+ + \sum_{\ell \in S} (q(i,\ell)-q_x(i,\ell))^+
\]
and
\[
0 = \sum_{\ell \in S} \bigl(q_x(i,\ell) - q(i,\ell)\bigr) = \sum_{\ell \in S} (q_x(i,\ell) - q(i,\ell))^+ - \sum_{\ell \in S} (q(i,\ell)-q_x(i,\ell))^+,
\]
calculation shows that $\Pr{[\eta_{n+1} = j \mid (X_n,\eta_n) = (x,i)]} = q_x(i,j)$ and $\Pr{[\eta^\star_{n+1} = j \mid \eta^\star_n = i]} = q(i,j)$. Then we see that
\[
\Pr{[ (X_{n+1}, \eta_{n+1}) = (y,j) \mid (X_n,\eta_n) = (x,i) ]} = \frac{ p(x,i,y,j)}{\sum_{z \in \ZP} p(x,i,z,j)}q_x(i,j) = p(x,i,y,j).
\]
This verifies the coupling construction. Note that, with this coupling,
\begin{equation}
\label{decoupling}
\Pr{[\eta_{n+1} \neq \eta^\star_{n+1} \mid X_n = x, \eta_n =\eta^\star_n = i]} = \frac{1}{2}\sum_{j \in S} |q_x(i,j) - q(i,j)|.
\end{equation}
It remains to prove \eqref{eqn:eta-eta-star-coupling} and \eqref{eqn:eta-eta-star-coupling2}. 
First in the case of \eqref{eqn:eta-eta-star-coupling}, for which we assume \eqref{ass:lim-q},
we give the argument in detail; we will then indicate how to modify the argument to prove \eqref{eqn:eta-eta-star-coupling2}.

Given $\eps >0$ and $n <\infty$, choose $x_0$ so that $\max_i \sum_{j \in S} |q_x(i,j) - q(i,j)| \leq \frac{\eps}{n}$ for all $x \geq x_0$;
this is possible by assumption \eqref{ass:lim-q}.

Let  $A_k = \{ X_k \geq x_0 \}$, and recall from \eqref{eq:E_n_def} that 
$E_k = \cap_{0 \leq \ell \leq k} \{ \eta_\ell = \eta^\star_\ell  \}$.
Then,
\[
\begin{split}
\Pr{[E_{k+1}^\rc \mid E_k \cap A_k ]} &= \Pr{[ \eta_{k+1} \neq \eta_{k+1}^\star \mid E_k \cap A_k ]}\\
&\leq \max_i \sup_{x\geq x_0} \Pr{[\eta_{k+1} \neq \eta_{k+1}^\star \mid X_k = x, \eta_k = \eta_k^\star = i ]},
\end{split}
\] so that, given $X_0 = x, \eta_0 = \eta^\star_0 = i$,
\[
\Pr{[ E^\rc_{k+1} ]} \leq \Pr{[ E^\rc_{k+1} \mid E_k \cap A_k]} + \Pr{[E^\rc_k]} + \Pr{[A^\rc_k]} \leq \frac{\eps}{2n} +\Pr{[E^\rc_k]} + \Pr{[A^\rc_k]},
\]
which in turn implies that
\[
\Pr{[E^\rc_n]} \leq \frac{\eps}{2} + \sum_{k=0}^{n-1} \Pr{[ A^\rc_k ]} \leq \frac{\eps}{2} + n \Pr{\left[\min_{0\leq k\leq n-1} X_k < x_0\right]}.
\]
To complete the proof we need to show that, for $x$ sufficiently large,
\begin{equation}\label{eqn:prob-X-small} \Pr{\left[\min_{0\leq k\leq n-1} X_k < x_0 \bigmid X_0 = x, \eta_0=\eta^\star_0 = i \right]} \leq \frac{\eps}{2n}, \text{ for all $i$}.
\end{equation}
But
\begin{align*}
\Pr \bigg[ \min_{0 \leq k \leq n-1} X_k < x_0 \bigmid   X_0 = x, & \, \eta_0=\eta^\star_0 = i \bigg]  \leq \Pr_{x,i} \bigg[ \max_{0 \leq k \leq n-1}| X_k - X_0 | > x-x_0  \bigg]
\\
& \leq \Pr_{x,i} {\bigg[ \bigcup_{0\leq k \leq n-1} |X_{k+1} - X_k| > \frac{x-x_0}{n}   \bigg]}\\
& \leq n \max_{y,j} \Pr{\bigg[ | X_{k+1} - X_k | > \frac{x-x_0}{n} \bigmid X_k = y ,\eta_k =j \bigg]},
\end{align*}
so \eqref{eqn:prob-X-small} will follow from 
\[
\lim_{r \to \infty} \max_{x, i} \Pr{[ |X_{n+1} - X_n| > r \mid X_n = x, \eta_n=i ]} = 0,
\]
which in turn follows from condition~\eqref{ass:X-diff-mom-p} with $p>1$ and Markov's inequality; indeed, 
\begin{align*}
\max_{x,i} \Pr{[ |X_{n+1} - X_n| > r \mid X_n=x, \eta_n=i ]} & \leq \max_{x,i} \frac{\Exp{[ |X_{n+1} - X_n|^{1+\eps} \mid X_n=x, \eta_n=i ]} }{r^{1+\eps}} \\
& \leq \frac{C_{1+\eps}}{r^{1+\eps}}.
\end{align*}
Therefore $\Pr{ [ E^\rc_n \mid X_0 = x, \eta_0 = \eta^\star_0 = i ]} \leq \eps$ for all $i$ and sufficiently large $x$, and since $\eps$ was arbitrary, this proves \eqref{eqn:eta-eta-star-coupling}.

The proof of \eqref{eqn:eta-eta-star-coupling2} is similar, now assuming \eqref{ass:lim-q+}. We set $n = n (x) = \lfloor A \log x \rfloor$.
Now we modify the definition of $A_k$ to be $A_k = \{ X_k \geq x/2 \}$.
Then, \eqref{decoupling} with \eqref{ass:lim-q+} gives
\[ 
\Pr{[ E^\rc_{k+1} ]} \leq  \Pr{[E^\rc_k]} + \Pr{[A^\rc_k]} + O (x^{-\delta_0}) ,\]
from which we have
\[ \Pr{[ E^\rc_{n(x)} ]} \leq O (x^{-\delta_0/2} ) + A \log x  \Pr{\left[\min_{0\leq k\leq n(x)-1} X_k < x / 2\right]} .\]
The final probability in the last display we estimate in exactly the same way as in the previous argument, replacing the previous $x_0$
by $x/2$ and the previous $n$ by $n(x)$, and we again find a term that decays as a power of $x$. Thus we obtain  \eqref{eqn:eta-eta-star-coupling2}.
\end{proof}

\section*{Acknowledgement}

This work was supported by the Engineering and Physical Sciences Research Council [grant number EP/J021784/1].

\end{document}